\documentclass[12pt]{article}
\usepackage{amsmath,amssymb,amsfonts,amsthm,graphicx,color}
\usepackage{mathrsfs}
\usepackage{hyperref}
\usepackage{mathtools}
\usepackage{graphicx,type1cm,eso-pic,color}
\usepackage[margin=1in]{geometry}
\allowdisplaybreaks
\title{Existence of global and explosive mild solutions of fractional reaction-diffusion system of semilinear SPDEs\\ with fractional noise}
\author{S. Sankar${}^{1}$, Manil T. Mohan${}^{2}$ and S. Karthikeyan${}^{1,}$\thanks{Corresponding author email: karthi@periyaruniversity.ac.in}\\
	\footnotesize{$^1$Department of Mathematics, Periyar University, Salem 636 011, India}\\
	\footnotesize{$^2$Department of Mathematics, Indian
	Institute of Technology Roorkee, Roorkee 247 667, India}}
\date{}
\allowdisplaybreaks

\usepackage{xcolor}
\allowdisplaybreaks

\usepackage[pagewise]{lineno}

\usepackage{graphicx,eurosym}
\usepackage{hyperref}
\usepackage{mathtools}

\colorlet{darkblue}{blue!50!black}

\hypersetup{
	colorlinks,%
	citecolor=blue,%
	filecolor=red,%
	linkcolor=darkblue,%
	urlcolor=blue,%
	pdfnewwindow=true,%
	pdfstartview={FitH}
}

\usepackage{graphicx,amscd,mathrsfs,wrapfig,mathrsfs,lipsum}
\usepackage{eufrak}
\usepackage{tikz}
\usepackage{multicol}
\usepackage{caption}
\usetikzlibrary{arrows}

\colorlet{darkblue}{red!100!black}

\let\originalleft\left
\let\originalright\right
\renewcommand{\left}{\mathopen{}\mathclose\bgroup\originalleft}
\renewcommand{\right}{\aftergroup\egroup\originalright}

\begin{document}
	\maketitle \setcounter{page}{1} \numberwithin{equation}{section}
	\newtheorem{theorem}{Theorem}[section]
	\newtheorem{assumption}{Assumption}
	\newtheorem{lemma}{Lemma}[section]
	\newtheorem{Pro}{Proposition}[section]
	\newtheorem{Ass}{Assumption}[section]
	\newtheorem{Def}{Definition}[section]
	\newtheorem{Ex}{Example}[section]
	\newtheorem{Rem}{Remark}[section]
	\newtheorem{corollary}{Corollary}[section]
	\newtheorem{proposition}{Proposition}[section] 
	\newtheorem{app}{Appendix:}
	\newtheorem{ack}{Acknowledgement:}

	\begin{abstract}
	In this paper, we investigate the existence and finite-time blow-up for the solution of a reaction-diffusion system of semilinear stochastic partial differential equations (SPDEs) subjected to a two-dimensional fractional Brownian motion given by 
	\begin{equation*}
	\left\{
	\begin{aligned}
	du_{1}(t,x)&=\left[ \Delta_{\alpha}u_{1}(t,x)+\gamma_{1}u_{1}(t,x)+u^{1+\beta_{1}}_{2}(t,x) \right]dt \\
	&\qquad \ \ +k_{11}u_{1}(t,x)dB^{H}_{1}(t)+k_{12}u_{1}(t,x)dB^{H}_{2}(t), \\
	du_{2}(t,x)&=\left[ \Delta_{\alpha}u_{2}(t,x)+\gamma_{2}u_{2}(t,x)+u^{1+\beta_{2}}_{1}(t,x) \right]dt \\
	&\qquad \ \ +k_{21}u_{2}(t,x)dB^{H}_{1}(t)+k_{22}u_{2}(t,x)dB^{H}_{2}(t), \\
	\end{aligned}
	\right.
	\end{equation*}
	for $x \in \mathbb{R}^{d},\ t \geq 0$, along with 	\begin{equation*}
	\begin{array}{ll}
	u_{i}(0,x)=f_{i}(x),  &x \in \mathbb{R}^{d}, \nonumber\\
	\end{array}
	\end{equation*}
	where $\Delta_{\alpha}$ is the fractional power $-(-\Delta)^{\frac{\alpha}{2}}$ of the Laplacian, $0<\alpha \leq 2$ and $\beta_{i}>0,\ \gamma_{i}>0$ and $k_{ij}\geq 0, i,j=1,2$ are constants. We provide sufficient conditions for the existence of a global weak solution. Under the assumption that $\beta_{1}\geq \beta_{2}>0$ with Hurst index $ 1/2 \leq H < 1,$ we obtain the blow-up times for an associated system of random partial differential equations  in terms of an integral representation of exponential functions of Brownian motions. Moreover, we provide lower and upper bounds for the finite-time blow-up of the above system of SPDEs and obtain the upper bounds for the probability of non-explosive solutions to our considered system.\\

	\noindent {\it Keywords: Semilinear SPDEs; fractional Brownian motion; blow-up times; stopping times; lower and upper bounds.}\\
		
	\noindent {\it MSC: 35R60; 60H15; 74H35}
	\end{abstract}
	\section{Introduction}
	A considerable effort has been devoted to exploring the problem of finite-time blow-up solutions to a numerous families of deterministic partial differential equations (PDEs). In particular, motivated by the primary and seminal papers by Kaplan \cite{kaplan} and Fujita \cite{fuji1966, Fuji1968}, new techniques were developed for studying the blow-up behavior to a more general class of nonlinear PDEs. As far as the stochastic counterpart is concerned, there are only a  limited results available on the non-existence of solutions of stochastic partial differential equations (SPDEs). For SPDEs driven by Brownian motions, Mueller and Sowers \cite{muller} proved that the heat equation	$$u_t=u_{xx}+u^\gamma\dot{W}$$ with continuous and nonnegative initial data blow-up in a finite-time with positive probability when $\gamma>3/2$ where 	$\dot{W}=\dot{W}(\cdot,\cdot)$ is a two-dimensional white noise process. Chow \cite{chow2009} proposed the notion of blow-up solution to nonlinear stochastic wave equations whose $L^2$-norms explode at a finite-time in the mean-square sense under some appropriate conditions on the initial data and the noise term. The existence of global positive solutions of semilinear stochastic partial differential equations (SPDEs) has been extensively investigated in \cite{chow09, chow11, friedman1985}, etc. Dozzi and  L\`opez-Mimbelab \cite{doz2010} investigated the  blow-up problem for the following SPDE:
    \begin{eqnarray}\label{basic}
   	du(t,x)=[\Delta  u(t,x)+G(u(t,x))]dt+\kappa u(t,x)dW(t), 
    \end{eqnarray}
 with the Dirichlet data,   on a smooth bounded domain, perturbed by a standard one-dimensional Brownian motion $\{W(t):t\geq 0\}$ defined on some probability space $\left( \Omega, \mathscr{F}, \mathbb{P} \right)$. Here $G:\mathbb R\to \mathbb R^+$ is a locally Lipschitz function satisfying
    \begin{eqnarray}\label{fung}
    G(z)\leq Cz^{1+\beta}, \ \mbox{for all} \ z>0,
    \end{eqnarray}  
    where $C \ \mbox{and} \ \beta >0$ are given numbers. The bounds for the explosion time are estimated in terms of exponential functionals of $W(\cdot)$ by transforming the original equation into a random PDE. For the case, $G(u)=u^{1+\beta}$ with $\beta>0$, the upper and lower bounds for the explosion time are estimated in \cite{car2013} by using a formula developed by Yor (cf. \cite{Fenman1997, revuz1999, yor2001}, etc.).     
    
    The impact of Gaussian noises on the time of blow-up of solutions to the following nonlinear SPDE:
    $$du(t,x)=[\Delta  u(t,x)+G(u(t,x))]dt+k_1u(t,x)dW_{1}(t)+k_2u(t,x)dW_2(t), $$    with the Dirichlet boundary condition, where $G$  satisfying \eqref{fung} and $\left( W_{1}(t),W_{2}(t) \right)_{t \geq 0} $ is a two-dimensional Brownian motion was investigated by Niu and Xie \cite{niu2012}. Dozzi et. al., \cite{doz2013} extended the results of \eqref{basic} to a system of semilinear SPDEs with the Dirichlet boundary data on a bounded smooth domain defined on some probability space $\left( \Omega, \mathscr{F}, \mathbb{P} \right)$.

 Concerning semilinear SPDEs driven by fractional Brownian motions,  Dozzi et. al., \cite{dozfrac2010} derived lower and upper bounds for the explosion time to a semilinear SPDE driven by a one-dimensional fractional Brownian motion $\{B^{H}(t)\}_{t\geq 0}$ with Hurst index $H>\frac{1}{2}$. Further, sufficient conditions for blow-up in finite-time and  the existence of a global solution are deduced and the probability of finite-time blow-up are obtained. Dung \cite{dung} proposed a new approach to obtain the upper and lower bounds for the probability of the finite-time blow-up of solution in terms of the Hurst parameter $H\in (0,1)$. Note that  for the case $H=\frac{1}{2}$,  $B^{H}(t)=W(t)$, a standard Brownian motion. In \cite{smk2}, lower and upper bounds for the explosion times of a system of semilinear SPDEs with standard Brownian motion are obtained. Sankar et. al., \cite{smk1} established the lower and upper bounds for the non-global solution for a system of semilinear SPDEs driven by two-dimensional fractional Brownian motion.
  
    Only a  few authors have addressed the explosive behavior of SPDEs with fractional Laplacian operators of the nonlocal type. Wang \cite{wang} proved the positivity of solutions to such a stochastic nonlocal heat equation under some assumptions and proved that the solution blow-up in  finite-time. In a recent paper, Dozzi et. al., \cite{doz2020} provided  sufficient conditions for finite-time blow-up of positive weak solutions of a fractional reaction-diffusion equation perturbed by a fractional Brownian motion (fBm). Moreover, lower and upper bounds for the probability that the solution does not explode in a finite-time are also provided. However, for many classes of SPDEs, the problem of existence and non-existence of solutions is not yet completely explored.
    
  	Motivated by the above facts, in the present paper, our aim is to obtain the global existence, and lower and upper bounds for the blow-up time to the following system of semilinear SPDEs:  
	\begin{equation}\label{b1}
	\left\{
	\begin{aligned}
	du_{1}(t,x)&=\left[ \Delta_{\alpha}u_{1}(t,x)+\gamma_{1}u_{1}(t,x)+u^{1+\beta_{1}}_{2}(t,x) \right]dt \\
	&\quad \ \ +k_{11}u_{1}(t,x)dB^{H}_{1}(t)+k_{12}u_{1}(t,x)dB^{H}_{2}(t), \\
	du_{2}(t,x)&=\left[ \Delta_{\alpha}u_{2}(t,x)+\gamma_{2}u_{2}(t,x)+u^{1+\beta_{2}}_{1}(t,x) \right]dt \\
	&\quad \ \ +k_{21}u_{2}(t,x)dB^{H}_{1}(t)+k_{22}u_{2}(t,x)dB^{H}_{2}(t), \\
	\end{aligned}
	\right.
	\end{equation}
	for $x \in \mathbb{R}^{d},\ t \geq 0$, along with initial conditions	\begin{equation}\label{b2}
	\begin{array}{ll}
	u_{i}(0,x)=f_{i}(x),  &x \in \mathbb{R}^{d},\ i=1,2, \nonumber
	\end{array}
	\end{equation}
	where $\Delta_{\alpha}$ is the fractional power $-(-\Delta)^{\frac{\alpha}{2}}$ of the Laplacian. Here, $0<\alpha \leq 2$, and $\beta_{1}\geq \beta_{2}>0,\ \gamma_{i}>0$ and $k_{ij}\geq 0, i,j=1,2$ are real constants. Moreover, $f_1,\ f_2$ are $C^{2}$-functions which are assumed to be non-negative and not identically zero. Further, $\left( B_{1}^{H}(t),B_{2}^{H}(t) \right)_{t \geq 0} $ is a standard independent two-dimensional fractional Brownian motion (fBm) having Hurst index $\frac{1}{2}\leq H<1$ with respect to a filtered probability space $\left( \Omega, \mathscr{F}, (\mathscr{F}_{t})_{t \geq 0}, \mathbb{P} \right)$. The operator $\Delta_{\alpha}$ is the usual Laplacian when  $\alpha=2$ and is the infinitesimal generator of $d$-dimensional Brownian motion with variance parameter 2. For $\alpha \in (0,2),$
    \begin{align} 
    \Delta_{\alpha}u(x)=C(d,\alpha) PV\left( \int_{\mathbb{R}^{d}} \frac{u(x+y)-u(x)}{\arrowvert y \arrowvert^{d+\alpha}} dy\right), \ \ u \in C^{2}_{c}(\mathbb{R}^{d}), \nonumber  
    \end{align} 
    where $C(d,\alpha)$ is a constant and $PV$ denotes the principal value of the integral, and in this case $\Delta_{\alpha}$ is the generator of the symmetric $\alpha$-stable L\'{e}vy process $\{{X}_t \}_{ t\geq 0}$ on $\mathbb{R}^{d}$ satisfying  $$\mathbb{E}\left[ \exp(i\boldsymbol{u}\cdot {X}_{t}) | {X}_{0}=0 \right]=\exp(-t |\boldsymbol{u} |^{\alpha}), \ \boldsymbol{u} \in \mathbb{R}^{d}, \ t \geq 0.$$ 
    
    The system \eqref{b1} is the two-dimensional version of the equation (1) studied in the paper of Dozzi et. al., \cite{doz2013} with $\alpha=2, H=1/2$ and also a generalization of the equation (1.3) in Dozzi et. al., \cite{doz2020}. Futher, the methods adopted in this paper closely follow the papers \cite{doz2013} and \cite{doz2020}.
    
    Let us recall the notion of weak and mild solutions of $\eqref{b1}$. Let $0\leq \tau \leq  +\infty$ be a stopping time. A continuous random vector field $u=(u_1,u_2)^{\top}=\left\lbrace (u_{1}(t,x),u_2(t,x))^{\top}: \ t\geq 0,\ x \in \mathbb{R}^{d} \right\rbrace $ is a weak solution of $\eqref{b1}$ on the interval $(0,\tau)$ provided that for every $\varphi_i\in C^\infty_{c}(\mathbb{R}^{d}),$ there holds
    \begin{align*} 
    	&\int_{D}u_{i}\left(t,x\right)\varphi_{i}(x)dx\\&=\int_{D}u_{i}\left(0,x\right)\varphi_{i}(x)dx+\int_{0}^{t}\int_{D}u_{i}\left(s,x\right)\Delta_{\alpha}\varphi_{i}(x)dxds+\gamma_{i}\int_{0}^{t}\int_{D}u_{i}\left(s,x\right)\varphi_{i}(x)dxds\nonumber\\
    	&\quad+\int_{0}^{t}\int_{D} u^{1+\beta_{i}}_{j}\left(s,x\right)\varphi_{i}(x)dxds+k_{i1}\int_{0}^{t}\int_{D}u_{i}\left(s,x\right)\varphi_{i}(x)dxdB^{H}_{1}(s)\nonumber\\
    	&\quad+k_{i2}\int_{0}^{t}\int_{D}u_{i}\left(s,x\right)\varphi_{i}(x)dxdB^{H}_{2}(s),\ \mathbb{P}-a.s.
    	\end{align*} for $i=1,2, \ j\in \left\{ 1,2 \right\} / \left\{i\right\}.$
    Let us denote by $\left\{ S_{t} \right\}_{t\geq 0}$, the $\alpha$-stable semigroup of bounded linear operators with generator $\Delta_{\alpha}$ and $\left\lbrace p(t,x): t>0, x\in \mathbb{R}^{d} \right\rbrace $ to indicate the family of spherically symmetric $\alpha$-stable transition densities. Hence, for any $t\geq 0$ and bounded measurable function $f : \mathbb{R}^{d} \rightarrow \mathbb{R}$, the bounded linear operator is defined by
    \begin{eqnarray}
    S_{t}f(x)=\mathbb{E}\left[ f(X_{t}) |X_{0}=x \right]=\int_{\mathbb{R}^{d}} p(t,y-x)f(y)dy, \quad  x\in \mathbb{R}^{d},\nonumber
    \end{eqnarray}
    where $\left\{ X_{t} \right\}_{t\geq 0}$ is the symmetric $\alpha$-stable L\'evy process  in $\mathbb{R}^{d}.$ 
    For every $t \geq 0,$ the linear operator $S_{t}$ is a contraction that preserves positivity and therefore the bounded linear operators
    \begin{align}\label{op1} 
    T_{i}(t)=\exp\{ \gamma_{i} t \}S_{t}, \ \ i=1,2,
    \end{align} 
    are also positivity preserving. Moreover, $\left\{T_{1}(t) \right\}_{t \geq 0 }$ and $\left\{T_{2}(t) \right\}_{t \geq 0 }$are linear semigroups  with infinitesimal generators $\Delta_{\alpha}+\gamma_{1}$ and $\Delta_{\alpha}+\gamma_{2}$ respectively. 
    
    From the semigroup theory \cite{pazy}, for any bounded measurable initial data $f_{i} \geq 0,$ a continuous random valued vector field $u=(u_1,u_2)^{\top}$ is a mild solution of the system \eqref{b1} if there exists a number $0<\tau=\tau(\omega)\leq \infty$ such that $u=(u_1,u_2)^{\top}$ satisfies the following integral equation on $(0, \tau)$:
    \begin{align}
    u_{i}(t,x)&=T_{i}(t) f_{i}(x)+\int_{0}^{t} T_{i}(t-s) u_{j}^{1+\beta_{i}}(s,x) ds+k_{i1}\int_{0}^{t} T_{i}(t-s)u_{i}(s,x)  dB^{H}_{1}(s) \nonumber\\
    &\qquad+k_{i2}\int_{0}^{t} T_{i}(t-s)u_{i}(s,x)  dB^{H}_{2}(s),\ \mathbb{P}\text{-a.s.} \nonumber 
    \end{align} for $i=1,2, \ j\in \left\{ 1,2 \right\} / \left\{i\right\}.$ From the classical results of semigroup theory \cite{pazy}, we note that for any bounded measurable initial data $f_{i} \geq 0, i=1,2,$ there exists a unique local mild solution $v=(v_1,v_2)^{\top}$ of the system of random PDEs \eqref{s1} (see below) by Theorem 2.5 in \cite{Atienza} which is the weak solution of the system \eqref{s1}. By using  the random transformation given in \eqref{tr1}, we obtain the existence and blow-up  of the weak solutions of the system \eqref{b1}. 
     
 The structure of the remaining paper is as follows:  In section \ref{sec2}, we obtain an associated system of random PDEs by using a random transformation and derive a weak solution  which is useful to obtain lower and upper bounds for the blow-up of solutions for the system \eqref{b1}.  Further, we prove that weak and mild solutions of  the system \eqref{b1} are equivalent. The lower bounds for the finite-time blow-up of the system \eqref{b1} with $1/2 < H<1$ and suitable parameters are obtained in section \ref{sec3} (cf. Theorem \ref{t2}). For appropriate initial data, upper bounds for the blow-up time of the solution $u=(u_1,u_2)^{\top}$ of the system \eqref{b1} for $1/2<H<1$, with the parameters $\rho_{1}=(1+\beta_{1})k_{21}-k_{11}=(1+\beta_{2})k_{11}-k_{21} $ and  $\rho_{2}=  (1+\beta_{1})k_{22}-k_{12}=(1+\beta_{2})k_{12}-k_{22}$ are obtained in section \ref{sec4} (cf. Theorem \ref{thm1}).  Furthermore, we establish upper bounds for the probability of non-explosive positive solution $u=(u_1,u_2)^{\top}$ of the system \eqref{b1} (cf. Theorem \ref{thm2}).  In section \ref{se12}, we obtain upper bounds for the probability of the existence of a global mild solution of the system \eqref{b1} for the case $H=\frac{1}{2}$ (Theorem \ref{thm5.1}). In section \ref{sec5}, we provide conditions for the global existence in time of the mild solution $u=(u_1,u_2)^{\top}$ to the system \eqref{b1} for the case $1/2\leq H<1$ and establish the upper bounds for the probability of the non-explosion solution $u=(u_1,u_2)^{\top}$ to the system \eqref{b1}. In the final section, we find lower and upper bounds for the blow-up time of the system \eqref{b1} for a more general case.

	\section{A System of Random PDEs}\label{sec2}
	In this section, we obtain a system of random PDEs by making use of the random transformations 
	\begin{align}\label{tr1}
	v_{i}(t,x) = \exp\{-k_{i1}B_{1}^{H}(t)-k_{i2}B_{2}^{H}(t)\}u_{i}(t,x),\ i=1,2,
	\end{align}
	for $t \geq 0,\ x \in \mathbb{R}^{d}$. The system $\eqref{b1}$ is transformed into the following  system of random PDEs:
	\begin{equation}\label{s1}
	\left\{
	\begin{aligned}
	\frac{\partial v_{i}(t,x)}{\partial t}&=\left( \Delta_{\alpha}+\gamma_{i} \right)v_{i}(t,x)+e^{-k_{i1}B_{1}^{H}(t)-k_{i2}B_{2}^{H}(t)}\left( e^{k_{j1}B_{1}^{H}(t)+k_{j2}B_{2}^{H}(t)}v_{j}(t,x) \right)^{1+\beta_{i}},\\
	v_{i}(0,x)&=f_{i}(x), \ \  x \in \mathbb{R}^{d},  
	\end{aligned}
	\right.
	\end{equation}
	for $i=1,2, \ j\in \left\{ 1,2 \right\} / \left\{i\right\}.$ This system is understood in the pathwise sense and thus the classical results for parabolic PDEs can be applied to show existence, uniqueness and positivity of solution $v=(v_1,v_2)^{\top}\in C([{0,\tau}];L^2(\mathbb{R}^d;\mathbb{R}^2))$ up to an eventual blow-up. 
	
	Next theorem helps us to prove that the system \eqref{b1} is transformed into the system of random PDEs \eqref{s1} when $\frac{1}{2}<H<1$.
	\begin{theorem} \label{thm2.1}
	Let $u=(u_1,u_2)^{\top}$ be the weak solution of $\eqref{b1}$. Then the function $v=(v_1,v_2)^{\top}$ defined by
	\begin{eqnarray}
	v_{i}(t,x) = \exp\{-k_{i1}B_{1}^{H}(t)-k_{i2}B_{2}^{H}(t)\}u_{i}(t,x), \ t \geq 0,\ x \in \mathbb{R}^{d},\ i=1,2, \nonumber 
	\end{eqnarray} 
    is a weak solution of the system of random PDEs $\eqref{s1}$ and viceversa.
		\end{theorem}
	\begin{proof}
	By using Ito's formula (Lemma 2.7.1 in \cite{mis2008}), we have 
	\begin{eqnarray}
	e^{-k_{i1}B^{H}_{1}(t)-k_{i2}B^{H}_{2}(t)}=1-\int_{0}^{t} e^{-k_{i1}B^{H}_{1}(s)-k_{i2}B^{H}_{2}(s)} (k_{i1}dB^{H}_{1}(s)+k_{i2}dB^{H}_{2}(s)),\ i=1,2. \nonumber 
	\end{eqnarray}
	For any smooth functions $\varphi_{1}$ and $\varphi_{2}$  with compact support, we denote
	$$u_{i}(t,\varphi_{i})=\int_{\mathbb{R}^{d}}u_{i}(t,x) \varphi_{i}(x)dx, \ \ i=1,2.$$
For every $\varphi_i\in C^\infty_{c}(\mathbb{R}^{d}),$ there holds on $(0,\tau)$
	\begin{align}\label{c1} 
	u_{i}\left(t,\varphi_{i}\right)&=u_{i}\left(0,\varphi_{i}\right)+{\int_{0}^{t}u_{i}\left(s,\Delta_{\alpha}\varphi_{i}\right)ds}+\gamma_{i}\int_{0}^{t}u_{i}\left(s,\varphi_{i}\right)ds+\int_{0}^{t} u^{1+\beta_{i}}_{j}\left(s,\varphi_{i}\right)ds\nonumber\\
	&\quad+k_{i1}\int_{0}^{t}u_{i}\left(s,\varphi_{i}\right)dB^{H}_{1}(s)+k_{i2}\int_{0}^{t}u_{i}\left(s,\varphi_{i}\right)dB^{H}_{2}(s),
	\end{align}
	where $i=1,2,\ \left\lbrace j \right\rbrace=\left\lbrace1,2\right\rbrace/\{i\}.$  By applying the integration by parts formula (see p.184 in \cite{mis2008}), we obtain
	\begin{align*}
	v_{i}(t,\varphi_{i}):&=\int_{\mathbb{R}^{d}}v_{i}(t,x)\varphi_{i}(x)dx \nonumber\\
    &=v_{i}(0,\varphi_{i})+\int_{0}^{t} e^{-k_{i1}B^{H}_{1}(t)-k_{i2}B^{H}_{2}(t)}du_{i}(s,\varphi_{i})\nonumber\\
	&\quad+\int_{0}^{t} u_{i}(s,\varphi_{i})\left(e^{-k_{i1}B^{H}_{1}(s)-k_{i2}B^{H}_{2}(s)} (k_{i1}dB^{H}_{1}(s)+k_{i2}dB^{H}_{2}(s))ds\right).\nonumber
	\end{align*} 
	Therefore,
	\begin{align}\label{a6}
	v_{i}(t,\varphi_{i}) &= v_{i}(0,\varphi_{i})+\int_{0}^{t}v_{i}\left(s,\Delta_{\alpha}\varphi_{i}\right)ds+\gamma_{i}\int_{0}^{t}v_{i}\left(s,\varphi_{i}\right)ds \nonumber\\
	&\qquad +\int_{0}^{t}e^{-k_{i1}B^{H}_{1}(t)-k_{i2}B^{H}_{2}(t)} \left( e^{k_{j1}B^{H}_{1}(t)+k_{j2}B^{H}_{2}(t)} v_{j} \right)^{1+\beta_{i}} (s,\varphi_i)ds \nonumber\\
	&= v_{i}(0,\varphi_{i})+\int_{0}^{t}[v_{i}\left(s,\Delta_{\alpha}\varphi_{i}\right)+\gamma_{i}v_{i}\left(s,\varphi_{i}\right)]ds \nonumber\\
	&\qquad+\int_{0}^{t}e^{-k_{i1}B^{H}_{1}(t)-k_{i2}B^{H}_{2}(t)} \left( e^{k_{j1}B^{H}_{1}(t)+k_{j2}B^{H}_{2}(t)} v_{j} \right)^{1+\beta_{i}} (s,\varphi_i)ds.  
	\end{align}
	The preceding equalities mean that $v=(v_1,v_2)^{\top}$ is a weak solution of $\eqref{s1}$. The converse part follows from the fact that the change of variable is given by a homeomorphism  which transforms one random dynamical system into an another equivalent one. 
	\end{proof}
	\begin{Rem}
	From the classical results available in semigroup theory \cite{pazy}, we note that for any bounded measurable initial data $f_{i} \geq 0,$ there exists a unique local mild solution $v=(v_1,v_2)^{\top}$ of the random PDE \eqref{s1} if there exists a number $0<\tau=\tau(\omega)\leq \infty$ such that $v=(v_1,v_2)^{\top}$ satisfies the following integral equation 
	\begin{align}\label{e3}
	v_{i}(t)=T_{i}(t) f_{i}(x)+\int_{0}^{t} T_{i}(t-r)\left[ e^{-k_{i1}B^{H}_{1}(r)-k_{i2}B^{H}_{2}(r)}\left(e^{k_{j1}B^{H}_{1}(r)+k_{j2}B^{H}_{2}(r)}v_{j}(r,\cdot)\right)^{1+\beta_{i}} \right](x)dr, 
	\end{align}
	or, equivalently,
	\begin{align}\label{n3}
	v_{i}(t,x)&=S_{t}f_{i}(x)\nonumber\\
	&\quad+\int_{0}^{t} S_{t-r}\left[\gamma_{i}v_{i}(s,\cdot)+ e^{-k_{i1}B^{H}_{1}(r)-k_{i2}B^{H}_{2}(r)}\left(e^{k_{j1}B^{H}_{1}(r)+k_{j2}B^{H}_{2}(r)}v_{j}(r,\cdot)\right)^{1+\beta_{i}} \right](x)dr,
	\end{align}
	for $i=1,2, \ j\in \left\{ 1,2 \right\} / \left\{i\right\}$ and for each $0 \leq t<\tau.$ 
	\end{Rem}

	
	\begin{lemma}\label{lem2.1}
	$v=(v_1,v_2)^{\top}$ is a weak solution of \eqref{s1} on $\left[0,\tau \right]$ if and only if $v=(v_1,v_2)^{\top}$ is a mild solution of \eqref{s1} on $\left[0,\tau \right].$  
	\end{lemma}
	\begin{proof}
	Assume that $v=(v_1,v_2)^{\top}$  is a   weak solution of \eqref{s1}. Let $h_{i} \in C^{1}(\mathbb{R}), \varphi_{i} \in C_{c}^{2}(\mathbb{R}^{d}),$ for $i=1,2$. For convenience, let us take $$g_{i,j}(t,x)=\gamma_{i}v_{i}(t,x)+e^{-k_{i1}B_{1}^{H}(t)-k_{i2}B_{2}^{H}(t)}\left( e^{k_{j1}B_{1}^{H}(t)+k_{j2}B_{2}^{H}(t)}v_{j}(t,x) \right)^{1+\beta_{i}},$$ for $i=1,2,  \ j\in \left\{ 1,2 \right\} / \left\{i\right\}.$
	Using integration by parts formula for any $t< \tau,$ and by denoting $ \langle\cdot,\cdot \rangle$,  the scalar product in $L^{2}(\mathbb{R}^d),$ we have
	\begin{align}
	\langle h_{i}(t) \varphi_{i}(\cdot),v_{i}(t,\cdot) \rangle&=\langle h_{i}(0) \varphi_{i}(\cdot),v_{i}(0,\cdot) \rangle + \int_{0}^{t} \langle h_{i}(s) \Delta_{\alpha}\varphi_{i}(\cdot),v_{i}(s,\cdot) \rangle ds \nonumber\\
	&\qquad+ \int_{0}^{t} \left\langle\frac{d}{ds} h_{i}(s) \varphi_{i}(\cdot),v_{i}(s,\cdot) \right\rangle ds + \int_{0}^{t} \langle h_{i}(s)\varphi_{i}(\cdot),g_{i,j}(s,\cdot) \rangle ds. \nonumber  
	\end{align} 
	Using a density argument, the following  is valid for $\Psi_{i}\in C^{1,2}_{c}\left(\mathbb{R} \times \mathbb{R}^{d} \right), \ i=1,2,$
	\begin{align}\label{n1}
	\langle \Psi_{i}(t,\cdot),v_{i}(t,\cdot)\rangle&=\langle\Psi_{i}(0,\cdot),v_{i}(0,\cdot)\rangle+\int_{0}^{t} \left\langle \frac{d}{ds}\Psi_{i}(s,\cdot)+\Delta_{\alpha}\Psi_{i}(s,\cdot),v_{i}(s,\cdot)\right \rangle ds\nonumber\\
	&\qquad +\int_{0}^{t} \langle \Psi_{i}(s,\cdot),g_{i,j}(s,\cdot) \rangle ds. 
	\end{align}
	Moreover, once again by a density argument, \eqref{n1} is valid for all $\Psi_{i}\in L^{1}\left([0,\tau] \times \mathbb{R}^{d} \right)$. For $\phi_{i}\in C_{c}^{2}(\mathbb{R}^{d})$, let $\Psi_{i},\ i=1,2$ be defined by
	\begin{align*}
	\Psi_{i}(s,x)=S_{t-s} \phi_{i}(x)=\left\{ \begin{array}{ccc}
	\langle p(t-s,\cdot-x),\phi_{i}(\cdot) \rangle, & \ \text{ if } \ \ s<t,\\\\
	\phi_{i}(x),& \ \text{ if } \ \ s=t.
	\end{array}\right.
	\end{align*}
	Since, $p(s,x)$ satisfies the Kolmogorov equation $\displaystyle\frac{d}{ds}p(s,x)=\Delta_{\alpha}p(s,x),$ we have 
	\begin{align}
	&\Delta_{\alpha}\Psi_{i}(s,x)+\frac{d}{ds}\Psi_{i}(s,x)\nonumber\\&=\Delta_{\alpha}\left( \int_{\mathbb{R}^{d}}p(t-s,y-x)\phi_{i}(y)dy \right)+\frac{d}{ds}\left( \int_{\mathbb{R}^{d}}p(t-s,y-x)\phi_{i}(y)dy \right)\nonumber\\
	&= \int_{\mathbb{R}^{d}} \left[\Delta_{\alpha}p(t-s,y-x)-\frac{d}{ds}p(t-s,y-x)\right]\phi_{i}(y)dy = 0. \nonumber   
	\end{align}
	Moreover, 
	\begin{align}
	\langle \Psi_{i}(0,\cdot),v_{i}(0,\cdot) \rangle&=\int_{\mathbb{R}^{d}} \Psi_{i}(0,x)v_{i}(0,x)dx=\int_{\mathbb{R}^{d}} f_{i}(x)\left( \int_{\mathbb{R}^{d}} p(t,y-x)\phi_{i}(y)dy \right)dx \nonumber\\ 
	&= \int_{\mathbb{R}^{d}} \phi_{i}(y)\left( \int_{\mathbb{R}^{d}} p(t,y-x)f_{i}(x)dx \right)dy \nonumber\\
	&=\int_{\mathbb{R}^{d}}\phi_{i}(y)\left( S_{t}f_{i} \right)(y)dy. \nonumber      
	\end{align}
	From \eqref{n1}, we have 
	\begin{align}
	\langle \phi_{i}(\cdot),v_{i}(t,\cdot) \rangle = \langle \phi_{i}(\cdot),S_{t}f_{i}(\cdot) \rangle+\int_{0}^{t} \langle \phi_{i}(\cdot),S_{t-s}g_{i,j}(s,\cdot) \rangle ds,\ i=1,2, \ j\in \left\{ 1,2 \right\} / \left\{i\right\}.\nonumber  
	\end{align}
    Since $v_{i}$ and $S_{t}v_{i}$ are locally integrable, and the  above equality holds for any $\phi_{i}\in C_{c}^{2}(\mathbb{R}^{d})$, it follows that $v_{i}$ solves \eqref{n3}. Hence $v=(v_1,v_2)^{\top}$ is a mild solution of \eqref{s1}.
		
	Conversely, assume that $v=(v_1,v_2)^{\top}$ is a mild solution of \eqref{s1}.   Let $\varphi_1,\varphi_2 \in C_{c}^{2}(\mathbb{R}^{d})$. Then  for any $t<\tau$, we have
	\begin{align}\label{n2}
	&\int_{0}^{t} \langle \varphi_{i}(\cdot),\Delta_{\alpha}v_{i}(s,\cdot) \rangle ds \nonumber\\&= \int_{0}^{t} \langle \Delta_{\alpha} \varphi_{i}(\cdot),v_{i}(s,\cdot) \rangle ds \nonumber\\
	&=\int_{0}^{t} \left\langle \Delta_{\alpha} \varphi_{i}(\cdot),S_{s}f_{i}(\cdot)+\int_{0}^{s} S_{s-r}g_{i,j}(r,\cdot)dr \right\rangle ds \nonumber\\
	&=\int_{0}^{t} \langle \Delta_{\alpha} \varphi_{i}(\cdot),S_{s}f_{i}(\cdot)\rangle ds + \int_{0}^{t}\left\langle \Delta_{\alpha} \varphi_{i}(\cdot), \int_{0}^{s} S_{s-r}g_{i,j}(r,\cdot)dr\right \rangle ds \nonumber\\
	&=\int_{0}^{t} \langle S_s\Delta_{\alpha} \varphi_{i}(\cdot),f_{i}(\cdot)\rangle ds + \int_{0}^{t}\left\langle \int_{r}^{t}S_{s-r}\Delta_{\alpha} \varphi_{i}(\cdot),g_{i,j}(r,\cdot)dr\right \rangle ds \nonumber\\
	&=\int_{0}^{t} \left\langle \frac{d}{ds}S_{s} \varphi_{i}(\cdot),f_{i}(\cdot)\right\rangle ds + \int_{0}^{t}\left\langle \int_{r}^{t}\frac{d}{ds}S_{s-r} \varphi_{i}(\cdot)ds,g_{i,j}(r,\cdot)\right\rangle dr \nonumber\\
	&= \langle S_{t} \varphi_{i}(\cdot),f_{i}(\cdot)\rangle-\langle \varphi_{i}(\cdot),f_{i}(\cdot)\rangle  + \int_{0}^{t}\langle S_{t-r} \varphi_{i}(\cdot),g_{i,j}(r,\cdot) \rangle dr-\int_{0}^{t}\langle \varphi_{i}(\cdot),g_{i,j}(r,\cdot) \rangle dr,
	\end{align} 
	for $i=1,2,\ j\in \left\{ 1,2 \right\} / \left\{i\right\}.$ Since $v=(v_1,v_2)^{\top}$ is a  mild solution of \eqref{s1}, we have 
	\begin{align}
	\langle \varphi_{i}(\cdot),v_{i}(t,\cdot) \rangle = \langle S_{t} \varphi_{i}(\cdot),f_{i}(\cdot) \rangle+\int_{0}^{t} \langle S_{t-s} \varphi_{i}(\cdot),g_{i,j}(s,\cdot) \rangle ds.\nonumber 
	\end{align} 
	From \eqref{n2}, we obtain
	\begin{align}
	\int_{0}^{t} \langle \varphi_{i}(\cdot),\Delta_{\alpha}v_{i}(s,\cdot) \rangle ds = \langle \varphi_{i}(\cdot),v_{i}(t,\cdot) \rangle-\langle \varphi_{i}(\cdot),f_{i}(\cdot)\rangle-\int_{0}^{t}\langle \varphi_{i}(\cdot),g_{i,j}(r,\cdot) \rangle dr. \nonumber
	\end{align}
	Therefore, 
	\begin{align}
	\langle \varphi_{i}(\cdot),v_{i}(t,\cdot) \rangle = \langle \varphi_{i}(\cdot),f_{i}(\cdot)\rangle+\int_{0}^{t} \langle \Delta_{\alpha} \varphi_{i}(\cdot),v_{i}(s,\cdot) \rangle ds +\int_{0}^{t}\langle \varphi_{i}(\cdot),g_{i,j}(r,\cdot) \rangle dr, \nonumber
	\end{align}
	for $i=1,2, \ j\in \left\{ 1,2 \right\} / \left\{i\right\}.$ Hence, $v=(v_1,v_2)^{\top}$ is a  weak solution of \eqref{s1}.
	\end{proof}
	
	
	Let $\tau$ be the blow-up time of the system \eqref{s1} with the initial values of the above form. Due to Theorem \ref{thm2.1} and a.s continuity of $B_{i}^{H}(\cdot), i=1,2$, $\tau$ is also the blow-up time for the system \eqref{b1} (Corollary 1 in \cite{dozfrac2010}).

    Next, we recall the following properties of $p(\cdot,\cdot)$  which will be useful for our further discussions (cf. Section 2 in \cite{sugi1975}).
    \begin{lemma}\label{l1}
    	For any $t>0,$ the function $p(t,\cdot)$ is continuous and strictly positive. Moreover, for any $t>0$ and $x\in \mathbb{R}^{d}$,
    	\begin{itemize}
    		\item [1.] $p(ts,x)=t^{-d/\alpha}p(s,t^{-1/\alpha}x),$ for any $s>0,$
    		\item [2.] $p(t,0) \geq p(t,x)$. More generally, $p(t,y) \geq p(t,x)$ if $\lVert x \rVert \geq \lVert y \rVert,$  
    		\item [3.] $p(t,x) \geq \left( \frac{s}{t}\right) ^{d/\alpha} p(s,x),$ for all $t \geq s,$
    		\item [4.] $p\left(t,\frac{1}{r}(x-y)\right) \geq p(t,x) p(t,y) , \ x, y \in \mathbb{R}^{d},$ provided that $p(t,0) \leq 1$ and $r \geq 2.$
    	\end{itemize}
    \end{lemma}
    In the sequel, $r_{0}>0$ denotes a fixed constant such that $p(r_{0},0)=r_{0}^{-d/\alpha}p(1,0)<1.$

	\section{A Lower Bound for $\tau$}\label{sec3}
	This section aims to find the lower bounds $\tau_{\ast}$ to the blow-up times such that $\tau_{\ast}\leq \tau$ when $\gamma_{1}=\gamma_{2}=\lambda.$ Hence $T_{1}(t)=T_{2}(t)=T(t)\ {\rm (say)}$  defined in \eqref{op1} and $\frac{1}{2}<H<1$. First, we consider the equation (\ref{b1}) with parameters $0<\alpha \leq 2$,  $\beta_{1}\geq \beta_{2}>0,\ \gamma_{i}>0$ and $k_{ij}\geq 0, i,j=1,2$ are real constants, and 
	\begin{equation}\label{a4}
	\begin{aligned} 
    (1+\beta_{1})k_{21}-k_{11}&=(1+\beta_{2})k_{11}-k_{21} =:\rho_{1},\\
	(1+\beta_{1})k_{22}-k_{12}&=(1+\beta_{2})k_{12}-k_{22}=:\rho_{2}.
	\end{aligned}
	\end{equation}
	The following result provides the lower bounds for the finite-time blow-up solution of the system \eqref{b1}.
	\begin{theorem}\label{t2} 
	Assume that the conditions (\ref{a4}) holds and $\frac{1}{2} < H<1.$ If the initial values are of the form $f_{1}=C_{1}\psi, f_{2}=C_{2}\psi,$	where $C_{1}$ and $C_{2}$ are any positive constants with $C_{1}\leq C_{2}$ and $\psi \in C_{c}^{\infty}({\mathbb R}^d).$ Then $\tau_{\ast}\leq\tau$, where $\tau_{\ast}$ is given by
	\begin{align}\label{t1}
	\tau_{\ast} = \inf \Bigg\{ t\geq 0 : &\int_{0}^{t} \exp\{\rho_{1} B^{H}_{1}(r)+\rho_{2}B^{H}_{2}(r)+\lambda \beta_{1}r\}r^{-d \beta_{1}/\alpha}dr 
	\geq \frac{1}{\beta_{1}p(1,0)^{\beta_{1}}C_{1}^{\beta_{1}}\left\|\psi \right\|_{\infty}^{\beta_{1}}}, \nonumber\\
(or) &\int_{0}^{t} \exp\{\rho_{1} B^{H}_{1}(r)+\rho_{2}B^{H}_{2}(r)+\lambda \beta_{2}r\}r^{-d \beta_{2}/\alpha}dr 
	\geq \frac{1}{\beta_{2}p(1,0)^{\beta_{2}}C_{2}^{\beta_{2}}\left\| \psi \right\|_{\infty}^{\beta_{2}}}	\Bigg\}. 
	\end{align}
	\begin{proof}
	Let $v=(v_{1},v_{2})^{\top}$ solve \eqref{e3}. Then, we have
	\begin{align}
	v_{1}(t,x)&=T(t)f_{1}(x)+\int_{0}^{t}T(t-r) \left( \exp\{{\rho_{1} B^{H}_{1}(r)+\rho_{2}B^{H}_{2}(r)}\}v^{1+\beta_{1}}_{2}(r,x)\right)dr,\nonumber\\
	v_{2}(t,x)&=T(t)f_{2}(x)+\int_{0}^{t}T(t-r) \left( \exp\{{\rho_{1} B^{H}_{1}(r)+\rho_{2}B^{H}_{2}(r)}\}v^{1+\beta_{2}}_{1}(r,x)\right)dr.\nonumber
	\end{align}
	For all $x\in \mathbb{R}^{d},  \ t\geq 0,$ let us define the operators $\mathcal{J}_{1}$ and $\mathcal{J}_{2}$ as follows:
	\begin{align}
	\mathcal{J}_{1}v(t,x)&=T(t)f_{1}(x)+\int_{0}^{t} \exp\{{\rho_{1} B^{H}_{1}(r)+\rho_{2} B^{H}_{2}(r)}\}\left(T(t-r)v\right)^{1+\beta_{1}}dr,\nonumber\\
	\mathcal{J}_{2} v(t,x)&=T(t)f_{2}(x)+\int_{0}^{t}\exp\{{\rho_{1} B^{H}_{1}(r)+\rho_{2} B^{H}_{2}(r)}\}\left(T(t-r)v\right)^{1+\beta_{2}}dr,\nonumber
	\end{align}
	where $v$ is any non-negative, bounded and measurable function. 
	
	First we shall prove that 
	\begin{eqnarray}
	v_{1}(t,x)=\mathcal{J}_{1} v_{2}(t,x), \  v_{2}(t,x)=\mathcal{J}_{2}v_{1}(t,x), \ \ x\in D, \ \ 0 \leq t<\tau_{*},\nonumber
	\end{eqnarray}
	for some non-negative, bounded and measurable functions $v_1$ and $v_2$. Moreover, on the set $t< \tau_{*},$ we set
	\begin{align}
	\mathscr{G}_{1}(t)&=\left[ 1-\beta_{1} \int_{0}^{t}  \exp\{\rho_{1} B^{H}_{1}(r)+\rho_{2} B^{H}_{2}(r)\}\left\| T(r)f_{1} \right\|_{\infty}^{\beta_{1}}dr\right]^{\frac{-1}{\beta_{1}}},\nonumber\\
	\mathscr{G}_{2}(t)&=\left[ 1- \beta_{2} \int_{0}^{t}  \exp\{\rho_{1} B^{H}_{1}(r)+\rho_{2} B^{H}_{2}(r)\}\left\| T(r)f_{2} \right\|_{\infty}^{\beta_{2}}dr\right]^{\frac{-1}{\beta_{2}}}.\nonumber
	\end{align}
	By \eqref{t1}, it is immediate that $\mathscr{G}_{1}(t)$ and $\mathscr{G}_{2}(t)$ are well defined. Then, it can be easily seen that 
	\begin{eqnarray}
	\frac{d \mathscr{G}_{1}(t)}{dt}=\exp\{\rho_{1} B^{H}_{1}(t)+\rho_{2} B^{H}_{2}(t)\}\left\| T(t)f_{1} \right\|_{\infty}^{\beta_{1}} \mathscr{G}^{1+\beta_{1}}_{1}(t), \ \ \mathscr{G}_{1}(0)=1, \nonumber
	\end{eqnarray}
	so that 
	\begin{eqnarray}
	\mathscr{G}_{1}(t)=1+\int_{0}^{t} \exp\{\rho_{1} B^{H}_{1}(r)+\rho_{2} B^{H}_{2}(r)\}  \left\| T(r)f_{1} \right\|_{\infty}^{\beta_{1}} \mathscr{G}^{1+\beta_{1}}_{1}(r) dr. \nonumber
	\end{eqnarray}
	Similarly, we have 
	\begin{eqnarray}
	\mathscr{G}_{2}(t)=1+\int_{0}^{t} \exp\{\rho_{1} B^{H}_{1}(r)+\rho_{2} B^{H}_{2}(r)\}  \left\| T(r)f_{2} \right\|_{\infty}^{\beta_{2}} \mathscr{G}^{1+\beta_{2}}_{2}(r) dr. \nonumber	
	\end{eqnarray}
	Let us choose $v\geq 0$ such that $$v(t,x)\leq T(t)f_{1}(x)\mathscr{G}_{1}(t),$$ for $x\in \mathbb{R}^{d}$ and $t<\tau_{*}.$ Then $T(t)f_{1}(x)\leq \mathcal{J}_{1} v(t,x)$ and
	\begin{align}
	\mathcal{J}_{1}v(t,x)&=T(t)f_{1}(x)+\int_{0}^{t}\exp\{{\rho_{1} B^{H}_{1}(r)+\rho_{2} B^{H}_{2}(r)}\}\left( T(t-r)v\right)^{1+\beta_{1}}dr\nonumber\\
	&\leq T(t)f_{1}(x)+\int_{0}^{t}\exp\{{\rho_{1} B^{H}_{1}(r)+\rho_{2} B^{H}_{2}(r)}\}\nonumber\\ 
	& \qquad\times\left[\mathscr{G}^{1+\beta_{1}}_{1}(r) \left\|T(r)f_{1}\right\|_{\infty}^{\beta_{1}}T(t-r)(T(r)f_{1}(x)) \right]dr \nonumber\\
	&=T(t)f_{1}(x)+\int_{0}^{t}\exp\{{\rho_{1} B^{H}_{1}(r)+\rho_{2} B^{H}_{2}(r)}\}\nonumber\\ 
	& \qquad\times\left[\mathscr{G}^{1+\beta_{1}}_{1}(r) \left\|T(r)f_{1}\right\|_{\infty}^{\beta_{1}}T(t)f_{1}(x) \right]dr \nonumber\\
	&= T(t)f_{1}(x)\Bigg\{ 1+ \int_{0}^{t} \exp\{{\rho_{1} B^{H}_{1}(r)+\rho_{2} B^{H}_{2}(r)}\}\left\| T(r)f_{1} \right\|_{\infty}^{\beta_{1}} \mathscr{G}^{1+\beta_{1}}_{1}(r)dr\Bigg\} \nonumber\\
	&= T(t)f_{1}(x)\mathscr{G}_{1}(t).\nonumber
	\end{align}
	Thus, we have
	\begin{eqnarray}\label{32}
	T(t)f_{1}(x)\leq \mathcal{J}_{1} v(t,x)  \leq T(t)f_{1}(x)\mathscr{G}_{1}(t).	
	\end{eqnarray}
	Similarly, we get 
	\begin{eqnarray}\label{33}
	T(t)f_{2}(x)\leq \mathcal{J}_{2} w(t,x)\leq T(t)f_{2}(x) \mathscr{G}_{2}(t), 
	\end{eqnarray}
	for all $w$ such that $0\leq w(t,x)\leq T(t)f_{2}(x)\mathscr{G}_{2}(t).$
	
	 Let us take,  
	\begin{eqnarray}\label{in1}
	u^{(0)}_{1}(t,x)=T(t)f_{1}(x), \ \ u^{(0)}_{2}(t,x)=\frac{C_{1}}{C_{2}}T(t)f_{2}(x), 
	\end{eqnarray}  
	and
	\begin{align}\label{in2}
	u^{(n)}_{1}(t,x)=\mathcal{J}_{1} u^{(n-1)}_{2}(t,x), \ \ u^{(n)}_{2}(t,x)=\mathcal{J}_{2} u^{(n-1)}_{1}(t,x), \ \ n\geq1, 
	\end{align}
	for $x\in \mathbb{R}^{d},\ 0\leq t < \tau_{*}.$ 
		
	Our aim is to show that the sequences of functions $\{ u^{(n)}_{1} \} \ \mbox{and} \ \{u^{(n)}_{2}\}$ are monotonically increasing. Let us consider
	\begin{align}
	u^{(0)}_{1}(t,x)&\leq T(t)f_{1}(x)+\int_{0}^{t} \exp\{{\rho_{1} B^{H}_{1}(r)+\rho_{2} B^{H}_{2}(r)}\} \left( T(t-r)u^{(0)}_{2}(r,x) \right)^{1+\beta_{1}}dr\nonumber\\
	&= \mathcal{J}_{1}u^{(0)}_{2}(t,x) = u^{(1)}_{1}(t,x).\nonumber
	\end{align}
	Now assume that $u^{(n)}_{i}\geq u^{(n-1)}_{i}, \ i=1,2,$ for some $n \geq 1.$  Then the monotonicity of $\mathcal{J}_{1}$  leads to the inequality
	\begin{eqnarray}
	u^{(n+1)}_{1}=\mathcal{J}_{1}  u^{(n)}_{2}(t,x)\geq \mathcal{J}_{1} u^{(n-1)}_{2}(t,x)=u^{(n)}_{1},\nonumber
	\end{eqnarray}
	and   the monotonicity of $\mathcal{J}_{2}$ results in 
	\begin{eqnarray}
	u^{(n+1)}_{2}=\mathcal{J}_{2}  u^{(n)}_{1}(t,x)\geq \mathcal{J}_{2} u^{(n-1)}_{1}(t,x)=u^{(n)}_{2}.\nonumber
	\end{eqnarray}
    Therefore the limits 
	\begin{eqnarray}
	v_{1}(t,x)=\lim_{n\rightarrow \infty}u^{(n)}_{1}(t,x), \ \ v_{2}(t,x)=\lim_{n\rightarrow \infty}u^{(n)}_{2}(t,x), \nonumber
	\end{eqnarray}
	exist for $x\in \mathbb{R}^{d}$ and $0 \leq t< \tau_{*}.$ Then by the monotone convergence theorem, we obtain 
	\begin{eqnarray}
	v_{1}(t,x)=\mathcal{J}_{1} v_{2}(t,x), \  v_{2}(t,x)=\mathcal{J}_{2} v_{1}(t,x), \ \ x\in \mathbb{R}^{d}, \ \ 0 \leq t<\tau_{*}.\nonumber
	\end{eqnarray}
	Since $f_{1}=C_{1} \psi$ and $f_{2}=C_{2} \psi,$ for some positive constants $C_{1}$ and $C_{2}$ with $C_{1} \leq C_{2}$, from \eqref{in1}, it follows that
		\begin{align}\label{eq2}
		\left\{
		\begin{aligned}
		0 &\leq u_{1}^{(0)}(t,x)=T(t)\left(\frac{C_{1}}{C_{2}} f_{2}(x)\right)=\frac{C_{1}}{C_{2}}T(t)f_{2}(x)\leq T(t)f_{2}(x) \mathscr{G}_{2}(t), \\
		0 &\leq \frac{C_{2}}{C_{1}} u_{2}^{(0)}(t,x)=T(t)\left(\frac{C_{2}}{C_{1}} f_{1}(x)\right)=\frac{C_{2}}{C_{1}}T(t) f_{1}(x), \\
		&u_{2}^{(0)}(t,x)=T(t)f_{1}(x)\leq T(t)f_{1}(x) \mathscr{G}_{1}(t).  
		\end{aligned}
		\right.   
		\end{align}
		From \eqref{in2}, we infer
		\begin{align}\label{eq3}
		u^{(1)}_{1}(t,x)&=\mathcal{J}_{1}u_{2}^{(0)}(t,x), 
		\end{align}
		so that  \eqref{eq2} gives 
		$u_{2}^{(0)}(t,x) \leq T(t) f_{1}(x)\mathscr{G}_{1}(t). $  	 
		Therefore from \eqref{eq3}, we obtain 
		\begin{align}\label{2}
		T(t)f_{1}(x) \leq u^{(1)}_{1}(t,x) \leq  T(t)f_{1}(x) \mathscr{G}_{1}(t).
		\end{align}
		Similarly, from \eqref{in2}, we deduce 
		\begin{align}\label{b11}
		u^{(1)}_{2}(t,x)=\mathcal{J}_{2}u^{(0)}_{1}(t,x), 
		\end{align}
		so that from \eqref{eq2}, we get  $u_{1}^{(0)}(t,x) \leq T(t)f_{2}(x)\mathscr{G}_{2}(t)$.		
		Therefore from \eqref{b11}, we have
		\begin{align}
		T(t)S_{t}f_{2}(x) \leq u^{(1)}_{2}(t,x) \leq  T(t)f_{2}(x) \mathscr{G}_{2}(t). \nonumber 
		\end{align}
		Note that \eqref{in2} leads to 
		\begin{align}\label{3}
		u^{(2)}_{1}(t,x)&=\mathcal{J}_{1}u^{(1)}_{2}(t,x), 
		\end{align}
		where $u^{(1)}_{2}(t,x)=\mathcal{J}_{2}u^{(0)}_{1}(t,x)$.
		We know that $ u_{1}^{(0)}(t,x) \leq T(t)f_{1}(x) \mathscr{G}_{1}(t)$, 
		so that 
		\begin{align}\label{9}
		u^{(1)}_{2}(t,x) \leq T(t)f_{1}(x)\mathscr{G}_{1}(t).
		\end{align}
		By using \eqref{9} in \eqref{3}, we have
		\begin{align}\label{4}
		T(t)f_{1}(x)\leq u^{(2)}_{1}(t,x) \leq T(t)f_{1}(x)\mathscr{G}_{1}(t).
		\end{align}
		Also  \eqref{in2} results in 
		\begin{align}\label{5}
		u^{(2)}_{2}(t,x)=\mathcal{J}_{2}u^{(1)}_{1}(t,x), 
		\end{align}
		where $u^{(1)}_{1}(t,x)=\mathcal{J}_{1}u^{(0)}_{1}(t,x)$. From \eqref{eq2} we know that $u^{(0)}_{1}(t,x) \leq T(t)f_{2}(x)\mathscr{G}_{2}(t)$.		
		Therefore $0 \leq u^{(1)}_{1}(t,x) \leq T(t)f_{2}(x) \mathscr{G}_{2}(t).$ 	
		From \eqref{5}, we obtain
		\begin{align}\label{6}
		T(t)f_{1}(x)\leq u^{(2)}_{2}(t,x) \leq T(t)f_{2}(x)\mathscr{G}_{2}(t).
		\end{align}
		Continuing like this,  for each $n \geq 1$,  one gets 
		\begin{equation*}
			\left\{
		\begin{aligned}
		T(t)f_{1}(x)&\leq u^{(n)}_{1}(t,x) \leq T(t)f_{1}(x)\mathscr{G}_{1}(t), \nonumber\\
		T(t)f_{2}(x)&\leq u^{(n)}_{2}(t,x) \leq T(t)f_{2}(x)\mathscr{G}_{2}(t). \nonumber
		\end{aligned}
	\right.
\end{equation*}
	Moreover,  we deduce 
	\begin{align}\label{v11}
	\left\{
	\begin{aligned}
	v_{1}(t,x) &\leq \frac{T(t)f_{1}(x)}{\left[ 1-\beta_{1}\int_{0}^{t}  \exp\{\rho_{1} B^{H}_{1}(r)+\rho_{2} B^{H}_{2}(r)\}\left\| T(r)f_{1} \right\|_{\infty}^{\beta_{1}}dr\right]^{\frac{1}{\beta_{1}}}}, \\
	v_{2}(t,x) &\leq \frac{T(t)f_{2}(x)}{\left[ 1-\beta_{2}\int_{0}^{t}  \exp\{\rho_{1} B^{H}_{1}(r)+\rho_{2} B^{H}_{2}(r)\}\left\| T(r)f_{2} \right\|_{\infty}^{\beta_{2}}dr\right]^{\frac{1}{\beta_{2}}}}.
	\end{aligned}
	\right. 	
	\end{align}
	By Lemma \ref{l1} (1), we have
	\begin{align}
	\left\|T(r)f_{i} \right\|_{\infty}^{\beta_{i}}=\left\|e^{\lambda r}S_{r}f_{i} \right\|_{\infty}^{\beta_{i}}&=e^{\lambda \beta_{i} r}\left( \sup_{x \in \mathbb{R}^{d}} \int_{\mathbb{R}^{d}}p(r,y-x)f_{i}(y)dy\right)^{\beta_{i}}\nonumber\\
	&=e^{\lambda \beta_{i} r}r^{-d \beta_{i}/\alpha} \left( \sup_{x \in \mathbb{R}^{d}} \int_{\mathbb{R}^{d}}p(1,r^{-1/\alpha}(y-x))f_{i}(y)dy\right)^{\beta_{i}}\nonumber\\
	&= e^{\lambda \beta_{i} r} r^{-d \beta_{i}/\alpha} p(1,0)^{\beta_{i}}\left\| f_{i} \right\|_{\infty}^{\beta_{i}}, \nonumber 
	\end{align}
	for $i=1,2$. Therefore,
	\begin{align}
	&\int_{0}^{t} \exp\{\rho_{1} B^{H}_{1}(r)+\rho_{2}B^{H}_{2}(r)\}\left\|T(r)f_{i} \right\|_{\infty}^{\beta_{i}}dr \nonumber\\
	&\qquad \qquad = p(1,0)^{\beta_{i}}\left\| f_{i} \right\|_{\infty}^{\beta_{i}}\int_{0}^{t} \exp\{\rho_{1} B^{H}_{1}(r)+\rho_{2}B^{H}_{2}(r)+\lambda \beta_{i}r\}r^{-d \beta_{i}/\alpha} dr.  \nonumber
	\end{align}
	From \eqref{v11}, we have
	\begin{align}
	v_{1}(t,x) &\leq \frac{T(t)f_{1}(x)}{\left[ 1-\beta_{1}p(1,0)^{\beta_{1}}\left\| f_{1} \right\|_{\infty}^{\beta_{1}}\int_{0}^{t}  \exp\{\rho_{1} B^{H}_{1}(r)+\rho_{2} B^{H}_{2}(r)+\lambda \beta_{1}r\}r^{-d \beta_{1}/\alpha}dr\right]^{\frac{1}{\beta_{1}}}}, \nonumber \\
	v_{2}(t,x) &\leq \frac{T(t)f_{2}(x)}{\left[ 1-\beta_{2}p(1,0)^{\beta_{2}}\left\| f_{2} \right\|_{\infty}^{\beta_{2}}\int_{0}^{t}  \exp\{\rho_{1} B^{H}_{1}(r)+\rho_{2} B^{H}_{2}(r)+\lambda \beta_{2}r\}r^{-d \beta_{2}/\alpha}dr\right]^{\frac{1}{\beta_{2}}}}. \nonumber 
	\end{align}
	By the choice of initial values, one can complete the proof.
	\end{proof}
    \end{theorem}
    	
	
	\begin{Rem}
	Since the initial data $f_1$ and $f_2$ are non-negative, the weak solutions $u=(u_1,u_2)^{\top}$ of \eqref{b1} and $v=(v_1,v_2)^{\top}$ of \eqref{s1} are non-negative as well. 
	\end{Rem}
	
	\section{Finite-time explosion of positive solutions}\label{sec4}
	In this section, we obtain  upper bounds for the blow-up time of the system \eqref{b1} with parameters  given in \eqref{a4}. Also we find upper bounds for the probability of non-explosive positive solution $u=(u_1,u_2)^{\top}$ of the system \eqref{b1} with $\frac{1}{2}<H<1$.
	\begin{lemma} \label{lem1}
	Suppose $v=(v_1,v_2)^{\top}$ is the mild solution of \eqref{s1}, $$ m_{i}(t)=\mathbb{E}\left[ v_{i}(t,X_{t})\right]=\int_{\mathbb{R}^{d}} p(t,y)v_{i}(t,y) dy=S_tv_{i}(t,0), \ \ i=1,2,$$
	where $\left\lbrace X_{t}, t \geq 0 \right\rbrace$ is a spherically symmetric $\alpha$-stable process in $\mathbb{R}^{d}$ starting from $0,$ and \eqref{a4} holds. For $\omega \in \Omega,$ if $m_{i}$ blows-up in a finite-time $\tau_{m_{i}}=\tau_{m_{i}}(\omega), \ i=1,2,$ then $v=(v_1,v_2)^{\top}$ blows-up in a finite-time $\tau_{v}$. Moreover,
	\begin{align}\label{n4}
	\tau_{v}=\min\{\tau_{v_1},\tau_{v_2}\} \leq \min \left\lbrace  (r_{0}+\tau_{m_{1}})(1+2^{\alpha}), (r_{0}+\tau_{m_{2}})(1+2^{\alpha}) \right\rbrace .
	\end{align}  
	\end{lemma}
	\begin{proof}
	For $i=1,2,\ \left\lbrace j \right\rbrace=\left\lbrace1,2\right\rbrace/\{i\}$, let $\tau_{m_{i}}<\infty$ be the blow-up time of $m_{i},$ and $t \geq \tau_{m_{j}}+2^{\alpha}(r_{0}+\tau_{m_{j}}).$ Then, for $s\in [0,\tau_{m_{j}}],$
	\begin{align}
	t-s \geq t-\tau_{m_{j}} \geq 2^{\alpha}(r_{0}+\tau_{m_{j}}) \geq 2^{\alpha}(r_{0}+s).\nonumber 
	\end{align}
	In the above inequality, we define
	\begin{align}\label{n5}
	\left( \frac{t-s}{r_{0}+s} \right)^{1/\alpha}=:\rho \geq 2. 
	\end{align}
	Since $p(r_{0},0)<1,$ we get by Lemma \ref{l1} (1)
	\begin{align}\label{n6}
	p(r_{0}+s,0)=p\left(\frac{r_0+s}{r_0}r_0,0\right)=\left( \frac{r_{0}+s}{r_{0}}\right)^{-d/\alpha}p(r_{0},0)<1, \ \mbox{ for all } \ s \in [0,\tau_{m_{j}}].
	\end{align}
	Therefore, for any $x,y \in \mathbb{R}^{d},$ by Lemma \ref{l1} (1), we obtain 
	\begin{align}
	p(t-s,y-x)&=p\left(\left(\frac{t-s}{r_{0}+s}\right) r_{0}+s,y-x \right)\nonumber\\
	&= \left( \frac{t-s}{r_{0}+s}\right)^{-d/\alpha}p\left( r_{0}+s,\left( \frac{t-s}{r_{0}+s}\right)^{-1/\alpha} (y-x)\right). \nonumber  
	\end{align}
	From \eqref{n5},\eqref{n6} and Lemma \ref{l1} (4), we deduce  the inequality
	\begin{align} \label{e5}
	p(t-s,y-x) &=\rho^{-d} p\left( r_{0}+s,\frac{y-x}{\rho} \right) 
	\geq \rho^{-d} p\left( r_{0}+s,x\right)p\left( r_{0}+s,y\right). 
	\end{align}
	Therefore, from \eqref{e3} and \eqref{e5}, it follows that for each $t \geq \tau_{m_{j}}+2^{\alpha}(r_{0}+\tau_{m_{j}}), \ i=1,2,\ \left\lbrace j \right\rbrace=\left\lbrace1,2\right\rbrace/\{i\}$, all $x \in \mathbb{R}^{d},$  
	\begin{align}\label{n7}
	&v_{i}(r_{0}+t,x)\nonumber\\&=\exp\{\gamma_{i}t\}\int_{\mathbb{R}^{d}}p(t,x-y)v_{i}(r_{0},y) dy \nonumber\\
	&\quad+\int_{0}^{t} \int_{\mathbb{R}^{d}} e^{ \gamma_{i}(t-s)}p(t-s,x-y)v_{j}^{1+\beta_{i}}(r_{0}+s,y)e^{ \rho_{1}B_{1}^{H}(r_{0}+s)+\rho_{2}B_{2}^{H}(r_{0}+s)}dy ds\nonumber \\
	& \geq \rho^{-d} \int_{0}^{\tau_{m_{j}}} \int_{\mathbb{R}^{d}} e^{ \gamma_{i}(t-s)+\rho_{1}B_{1}^{H}(r_{0}+s)+\rho_{2}B_{2}^{H}(r_{0}+s)} p\left( r_{0}+s,x\right)p\left( r_{0}+s,y\right)v_{j}^{1+\beta_{i}}(r_{0}+s,y)dy ds \nonumber\\
	&= \rho^{-d} \int_{0}^{\tau_{m_{j}}} e^{ \gamma_{i}(t-s)+\rho_{1}B_{1}^{H}(r_{0}+s)+\rho_{2}B_{2}^{H}(r_{0}+s)} p\left( r_{0}+s,x\right)\left[\int_{\mathbb{R}^{d}} p\left( r_{0}+s,y\right)v_{j}^{1+\beta_{i}}(r_{0}+s,y)dy  \right] ds. 
	\end{align}  
	Here $p(r_{0}+s,y)$ is the probability density and by using Jensen's inequality, we have
	\begin{align}
	\left[\int_{\mathbb{R}^{d}} p\left( r_{0}+s,y\right)v^{1+\beta_{i}}_{j}(r_{0}+s,y)dy  \right] &\geq \left[\int_{\mathbb{R}^{d}} p\left( r_{0}+s,y\right)v_{j}(r_{0}+s,y)dy  \right]^{1+\beta_{i}} = m_{j}^{1+\beta_{i}}(r_{0}+s). \nonumber 
	\end{align}
	Therefore \eqref{n7} becomes
	\begin{align}
	v_{i}(r_{0}+t,x)& \geq \rho^{-d} \int_{0}^{\tau_{m_{j}}} e^{ \gamma_{i}(t-s)+\rho_{1}B_{1}^{H}(r_{0}+s)+\rho_{2}B_{2}^{H}(r_{0}+s)} p\left( r_{0}+s,x\right) m_{j}^{1+\beta_{i}}(r_{0}+s) ds\nonumber\\&=  \rho^{-d} \int_{r_0}^{\tau_{m_{j}}} e^{ \gamma_{i}(t+r_0-\ell)+\rho_{1}B_{1}^{H}(\ell)+\rho_{2}B_{2}^{H}(\ell)} p\left(\ell,x\right) m_{j}^{1+\beta_{i}}(\ell) d\ell\nonumber\\&\quad+\rho^{-d} \int_{\tau_{m_{j}}}^{r_0+\tau_{m_{j}}} e^{ \gamma_{i}(t+r_0-\ell)+\rho_{1}B_{1}^{H}(\ell)+\rho_{2}B_{2}^{H}(\ell)} p\left(\ell,x\right) m_{j}^{1+\beta_{i}}(\ell) d\ell. \nonumber
	\end{align}
	Since $r_{0}>0,$ the second integral in the RHS of the above inequality is infinite. Hence $\tau_{v_i} \leq r_{0}+\tau_{m_{j}}+2^{\alpha}(r_{0}+\tau_{m_{j}})=(r_{0}+\tau_{m_{j}})(1+2^{\alpha}),$ for $i=1,2,\ \left\lbrace j \right\rbrace=\left\lbrace1,2\right\rbrace/\{i\}$, 
	which  completes the proof.
	\end{proof}


	The following theorem gives an upper bound for the finite-time blow-up of solution $u=(u_1,u_2)^{\top}$ of the system \eqref{b1}.
	\begin{theorem}\label{thm1}
	Let $\frac{1}{2}<H<1$ and let $u=(u_1,u_2)^{\top}$ be a weak solution of \eqref{b1} with nonnegative initial data $f_1$ and $f_2$ which are  bounded measurable functions and \eqref{a4} holds. 	
	\item [1.] If $\beta_{1}=\beta_{2}=\beta>0.$ Then $u=(u_1,u_2)^{\top}$ blows-up in finite-time which is  given by
	\begin{align}
	\theta_{1} = \inf \left\lbrace t>r_{0} : \int_{r_{0}}^{t} \min \{ g_{1,2}(s),g_{2,1}(s) \} ds \geq \frac{2^{1+\beta}}{\beta \left(r_{1} e^{-\gamma_{1}r_{0}}+r_{2}e^{-\gamma_{2}r_{0}} \right)^{\beta}}    \right\rbrace. \nonumber
	\end{align}
	Moreover, if $\tau$ is the blow-up time of \eqref{b1}, then \begin{align}
	\tau &\leq \min \left\lbrace  (r_{0}+\tau_{m_1})(1+2^{\alpha}), (r_{0}+\tau_{m_2})(1+2^{\alpha}) \right\rbrace \leq (r_{0}+\theta_{1})(1+2^{\alpha}), \nonumber
	\end{align}	
	where $r_{i}=2^{-2d}p(1,0)\exp\{\gamma_{i}r_{0}\}\mathbb{E}\left[ f_{i}(X_{2^{-\alpha}r_{0}}) \right], \ i=1,2$ and $$g_{i,j}(s)=2^{\frac{-d(1+\beta_{i})}{\alpha}}e^{k_{i,j}s+\rho_{1}B_{1}^{H}(s)+\rho_{2}B_{2}^{H}(s)}s^{\frac{-d\beta_{i}}{\alpha}},$$ where $\rho_{1}, \rho_{2}$ are defined in \eqref{a4} and $k_{i,j}=-\gamma_{i}+(1+\beta_{i})\gamma_{j}$ for $i=1,2, \ j\in \left\{ 1,2 \right\} / \left\{i\right\}.$  
	\item [2.] If $\beta_{1}>\beta_{2}>0$ and let $D_{1}=\left( \frac{\beta_{1}-\beta_{2}}{1+\beta_{1}} \right)\left(\frac{1+\beta_{1}}{1+\beta_{2}} \right)^{\frac{1+\beta_{2}}{\beta_{1}-\beta_{2}}},$
	\begin{align}\label{nca3}
	\epsilon_{0}&\leq \min \Bigg\{ 1, \left( h_{2}(r_0)/D_{1}^{1/1+\beta_{2}} \right)^{\beta_{1}-\beta_{2}} \Bigg\}.
	\end{align} 
	Assume that 
	\begin{align}\label{nca4}
	2^{-(1+\beta_{2})}\epsilon_{0}\left(r_{1} e^{-\gamma_{1}r_{0}}+r_{2}e^{-\gamma_{2}r_{0}}\right)^{1+\beta_{2}}\geq \epsilon_{0}^{\frac{1+\beta_{1}}{\beta_{1}-\beta_{2}}}D_{1}.
	\end{align} 
    Then $u=(u_1,u_2)^{\top}$ blow-up in finite-time which is given by
	\begin{align}
	\theta_{2} = &\inf \Bigg\{  t> r_0 : \int_{r_0}^{t} \min\left\lbrace g_{1,2}(s),g_{2,1}(s) \right\rbrace ds \geq \nonumber\\ 
	& \qquad\left[\beta_{2} \left(r_{1} e^{-\gamma_{1}r_{0}}+r_{2}e^{-\gamma_{2}r_{0}}\right)^{\beta_{2}} \left(\frac{\epsilon_0}{2^{1+\beta_{2}}}-\frac{ \epsilon_{0}^{\frac{1+\beta_{1}}{\beta_{1}-\beta_{2}}}D_{1}}{\left(r_{1} e^{-\gamma_{1}r_{0}}+r_{2}e^{-\gamma_{2}r_{0}}\right)^{1+\beta_{2}}}\right)\right]^{-1} \Bigg\}, \nonumber
	\end{align}
	Moreover, if $\tau$ is the blow-up time of \eqref{b1}, then  
	\begin{align}
	\tau &\leq \min \left\lbrace  (r_{0}+\tau_{m_1})(1+2^{\alpha}), (r_{0}+\tau_{m_2})(1+2^{\alpha}) \right\rbrace \leq (r_{0}+\theta_{2})(1+2^{\alpha}), \nonumber
	\end{align} 
	where $r_{i}=p(1,0)2^{-2d}\exp\{\gamma_{i}r_{0}\}\mathbb{E}\left[ f_{i}(X_{2^{-\alpha}r_{0}}) \right], \ i=1,2$ and for any $r>0,$ 
	\begin{align}\label{48}
	g_{i,j}(s)=2^{\frac{-d(1+\beta_{i})}{\alpha}}e^{k_{i,j}s+\rho_{1}B_{1}^{H}(s)+\rho_{2}B_{2}^{H}(s)}s^{\frac{-d\beta_{i}}{\alpha}},
	\end{align} where $\rho_{1}, \rho_{2}$ are defined in \eqref{a4} and 
    \begin{align}\label{49}
	k_{i,j}=-\gamma_{i}+(1+\beta_{i})\gamma_{j}\ \text{ for } \ i=1,2, \ j\in \left\{ 1,2 \right\} / \left\{i\right\}. 
    \end{align}
	\end{theorem}

	\begin{proof}
	By using the property of semigroup for $v_{i}, \ i=1,2$, we have
	\begin{align*}
	v_{i}(t+r_{0})=T_{i}(t)v_{i}(r_{0})+\int_{0}^{t} T_{i}(t-s)v_{j}^{1+\beta_{i}}(r_{0}+s) e^{\rho_{1}B_{1}^{H}(r_{0}+s)+\rho_{2}B_{2}^{H}(r_{0}+s)} ds, \ t \geq 0,
	\end{align*}
	for $j\in \left\{ 1,2 \right\} / \left\{i\right\}.$	Therefore 
	\begin{align}\label{nn1}
	&m_{i}(t+r_{0})\nonumber\\&=S_{t+r_{0}}v_{i}(t+r_{0})(0) \nonumber \\
	&=S_{t+r_{0}}\left( e^{\gamma_{i}t}S_{t}v_{i}(r_{0})+\int_{0}^{t} e^{\gamma_{i}(t-s)}S_{t-s}v_{j}^{1+\beta_{i}}(r_{0}+s) e^{\rho_{1}B_{1}^{H}(r_{0}+s)+\rho_{2}B_{2}^{H}(r_{0}+s)} ds \right)(0) \nonumber \\
	&= S_{2t+r_{0}}e^{\gamma_{i}t}v_{i}(r_{0},0)+\int_{0}^{t} e^{\gamma_{i}(t-s)}S_{2t+r_{0}-s}v_{j}^{1+\beta_{i}}(r_{0}+s,0) e^{\rho_{1}B_{1}^{H}(r_{0}+s)+\rho_{2}B_{2}^{H}(r_{0}+s)} ds.
	\end{align}
	By using Lemma \ref{l1} (1) and Jensen's inequality, we obtain
	\begin{align}\label{n8}
	S_{2t+r_{0}-s}v_{j}^{1+\beta}(r_{0}+s,0) &\geq \left( \frac{r_{0}+s}{2t+r_{0}-s}\right)^{d/\alpha}S_{r_{0}+s}v_{j}^{1+\beta_{i}}(r_{0}+s,0) \nonumber\\  
	& \geq 	\left( \frac{r_{0}+s}{2t+2r_{0}}\right)^{d/\alpha}\left( S_{r_{0}+s}v_{j}(r_{0}+s,0)\right)^{1+\beta_{i}}  \nonumber\\
	&= \left( \frac{r_{0}+s}{2t+2r_{0}}\right)^{d/\alpha}m_{j}^{1+\beta_{i}}(r_{0}+s), \ 0 \leq s \leq t,
	\end{align}
	for $i=1,2, \ j\in \left\{ 1,2 \right\} / \left\{i\right\}.$ Also, we have  
	\begin{align}
	v_{i}(r_{0},y) \geq \exp\{ \gamma_{i} r_{0} \} \int_{\mathbb{R}^{d}} p(r_{0},x-y)f_{i}(x) dx , \nonumber
	\end{align}
	for all $y \in \mathbb{R}^{d}. $ Since $p(r_{0},0)<1$ and by Lemma \ref{l1} (1) and (4),  we get 
	\begin{align}
	\exp\{\gamma_{i}r_{0}\} &\int_{\mathbb{R}^{d}} p\left( r_{0}, \frac{2x-2y}{2}\right)f_{i}(x)dx  \nonumber\\&\geq \exp\{\gamma_{i}r_{0}\} p(r_{0},2y)\int_{\mathbb{R}^{d}}p(r_{0},2x)f_{i}(x)dx \nonumber\\
	&= 2^{-2d}\exp\{\gamma_{i}r_{0}\}p(2^{-\alpha}r_{0},y) \int_{\mathbb{R}^{d}}p(2^{-\alpha}r_{0},x) f_{i}(x)dx \nonumber\\
	&= 2^{-2d}\exp\{\gamma_{i}r_{0}\}\mathbb{E}\left[ f_{i}(X_{2^{-\alpha}r_{0}}) \right]  p(2^{-\alpha}r_{0},y),\nonumber
	\end{align}
	for $i=1,2.$ Therefore we have $$v_{i}(r_{0},y)\geq 2^{-2d}\exp\{\gamma_{i}r_{0}\}\mathbb{E}\left[ f_{i}(X_{2^{-\alpha}r_{0}}) \right] p(2^{-\alpha}r_{0},y) .$$ 
	By Lemma \ref{l1} (4) and semigroup property, we deduce 
	\begin{align}\label{n9}
	S_{2t+r_{0}}v_{i}(r_{0},0)&=\int_{\mathbb{R}^{d}}p(2t+r_{0},y)v_{i}(r_{0},y)dy\nonumber\\
	&\geq 2^{-2d}\exp\{\gamma_{i}r_{0}\}\mathbb{E}\left[ f_{i}(X_{2^{-\alpha}r_{0}}) \right]p\left( 2t+(1+2^{-\alpha})r_{0},0\right) \nonumber\\&=2^{-2d}\exp\{\gamma_{i}r_{0}\}\mathbb{E}\left[ f_{i}(X_{2^{-\alpha}r_{0}}) \right]\left(2t+(1+2^{-\alpha})r_{0}\right)^{-d/\alpha}p(1,0)\nonumber\\
	&\geq 2^{-2d}\exp\{\gamma_{i}r_{0}\}\mathbb{E}\left[ f_{i}(X_{2^{-\alpha}r_{0}}) \right](2t+2r_{0})^{-d/\alpha}p(1,0)\nonumber\\
	& =r_{i}(2t+2r_{0})^{-d/\alpha},
	\end{align}
	where 
    \begin{align} \label{nn2}
	r_{i}=2^{-2d}p(1,0)\exp\{\gamma_{i}r_{0}\}\mathbb{E}\left[ f_{i}(X_{2^{-\alpha}r_{0}}) \right], \ i=1,2.
	\end{align}
	Plugging \eqref{n8} and \eqref{n9} into \eqref{nn1}, and multiplying both sides of \eqref{nn1} by $\exp\{- \gamma_{i}(r_{0}+t)\}(2t+2r_{0})^{d/\alpha}$ for $i=1,2,$ we have 
	\begin{align}
	m_{i}(t+r_{0}) &\geq r_{i}(2t+2r_{0})^{-d/\alpha}e^{\gamma_{i}t}\nonumber\\
	&\quad+\int_{0}^{t}e^{\gamma_{i}(t-s)}\left( \frac{r_{0}+s}{2t+2r_{0}}\right)^{d/\alpha}m_{j}^{1+\beta_{i}}(r_{0}+s) e^{\rho_{1}B_{1}^{H}(r_{0}+s)+\rho_{2}B_{2}^{H}(r_{0}+s)} ds. \nonumber
	\end{align}
	That is, 
	\begin{align*}
	&\exp\{- \gamma_{i}(r_{0}+t)\}(2t+2r_{0})^{d/\alpha}m_{i}(t+r_{0})\\
	&\geq e^{-\gamma_{i}r_{0}}r_{i}+\int_{0}^{t} e^{-\gamma_{i}(r_{0}+s)}(r_{0}+s)^{d/\alpha}m_{j}^{1+\beta_{i}}(r_{0}+s) e^{\rho_{1}B_{1}^{H}(r_{0}+s)+\rho_{2}B_{2}^{H}(r_{0}+s)}ds.\nonumber		
	\end{align*}
	Let 
	\begin{align} \label{ne1}
	w_{i}(r)&=e^{-\gamma_{i}r}(2r)^{d/\alpha}m_{i}(r), \ \  \mbox{and} \nonumber\\ g_{i,j}(r)&=2^{\frac{-d(1+\beta_{i})}{\alpha}}e^{k_{i,j}r+\rho_{1}B_{1}^{H}(r)+\rho_{2}B_{2}^{H}(r)}r^{\frac{-d\beta_{i}}{\alpha}}, \ \ r>0,
	\end{align}
	where $k_{i,j}=-\gamma_{i}+(1+\beta_{i})\gamma_{j}$ for $i=1,2, \ j\in \left\{ 1,2 \right\} / \left\{i\right\}.$ Then, we  have  
	\begin{align}
	w_{i}(r_{0}+t) &=e^{-\gamma_{i}(r_0+t)}(2(r_0+t))^{d/\alpha}m_{i}(r_0+t)\nonumber\\&\geq r_{i}e^{-\gamma_{i,j}r_{0}}+\int_{0}^{t}g_{i,j}(r_{0}+s)w_{j}^{1+\beta_{i}}(r_{0}+s) ds \nonumber\\
	&=r_{i}e^{-\gamma_{i}r_{0}}+\int_{r_{0}}^{r_{0}+t}g_{i,j}(r)w_{j}^{1+\beta_{i}}(r) dr, \nonumber		
	\end{align} 
	which can be equivalently written as 
	\begin{align}
	w_{i}(t) &\geq r_{i}e^{-\gamma_{i}r_{0}}+\int_{r_{0}}^{t}g_{i,j}(r)w_{j}^{1+\beta_{i}}(r) dr, \ \ t \geq r_{0}. \nonumber
	\end{align} 
	Notice that $m_{i}$ and $w_{i}$ explode at the same finite-time due to \eqref{ne1}, and so do $v_{i}$ and $u_{i}.$ Moreover, by comparison argument (see Theorem 1.3 of \cite{teschl} and Appendix), we have $w_{i}(t) \geq h_{i}(t), \ i=1,2 $ for all $ t \geq r_{0},$ where $h_1$ and $h_2$ are given by
	\begin{align*}
	h_{1}(t)&=r_{1}e^{-\gamma_{1}r_{0}}+\int_{r_{0}}^{t}g_{1,2}(r)h_{2}^{1+\beta_{1}}(r) dr, \\
	h_{2}(t)&=r_{2}e^{-\gamma_{2}r_{0}}+\int_{r_{0}}^{t}g_{2,1}(r)h_{1}^{1+\beta_{2}}(r) dr. 
	\end{align*}
	Define $E(t)=h_{1}(t)+h_{2}(t),\ t \geq r_{0},$ so that 
	\begin{align*}
	E(t)=r_{1} e^{-\gamma_{1}r_{0}}+r_{2}e^{-\gamma_{2}r_{0}} +\int_{r_{0}}^{t}g_{1,2}(r)h_{2}^{1+\beta_{1}}(r) dr+\int_{r_{0}}^{t}g_{2,1}(r)h_{1}^{1+\beta_{2}}(r) dr.\nonumber 
	\end{align*}
	Therefore, 
	\begin{equation}\label{eq1}
	\left\{
	\begin{aligned}
	\frac{dE(t)}{dt}&=g_{1,2}(t)h_{2}^{1+\beta_{1}}(t)+g_{2,1}(t)h_{1}^{1+\beta_{2}}(t),\\
	E(r_{0})&=r_{1}e^{-\gamma_{1}r_{0}}+r_{2}e^{-\gamma_{2}r_{0}}.  
	\end{aligned}
    \right. 
    \end{equation}
	\textbf{Case 1:} Suppose that $\beta_{1}=\beta_{2}=\beta>0$ in \eqref{eq1}. Then,  we have
	\begin{align}\label{sa1}
	\frac{dE(t)}{dt}&= g_{1,2}(t)h_{2}^{1+\beta}(t)+g_{2,1}(t)h_{1}^{1+\beta}(t)\nonumber \\
	& \geq \min\left\lbrace g_{1,2}(t),g_{2,1}(t) \right\rbrace  \left[ h_{2}^{1+\beta}(t)+h_{1}^{1+\beta}(t)\right]  .
	\end{align}
	By substituting $a=1, \ b=\frac{h_{1}}{h_{2}},\ n=1+\beta$ into the inequality 
	\begin{eqnarray}\label{cta1}
	a^{n}+b^{n}\geq 2^{-n}{(a+b)}^{n},
	\end{eqnarray}
	which is valid for $a, \ b\geq0,$ we obtain 
	\begin{eqnarray}
	&&h_{1}^{1+\beta}(t)+h_{2}^{1+\beta}(t) \geq 2^{-(1+\beta)} \left( h_{1}(t)+h_{2}(t) \right)^{1+\beta} \geq 2^{-(1+\beta)} E^{1+\beta}(t).\nonumber
	\end{eqnarray}
	Form \eqref{sa1}, we have
	\begin{align}
	\frac{dE(t)}{dt} \geq \min\left\lbrace g_{1,2}(t),g_{2,1}(t) \right\rbrace  2^{-(1+\beta)} E^{1+\beta}(t) .\nonumber
	\end{align} 
	Thus, by comparison argument (see Theorem 1.3 of \cite{teschl} and Appendix), we get $E(t) \geq I(t)$ for all $t \geq 0$, where $I(t)$ solves the differential equation
	\begin{align}
	\frac{dI(t)}{dt} = \min\left\lbrace g_{1,2}(t),g_{2,1}(t) \right\rbrace  2^{-(1+\beta)} I^{1+\beta}(t), \ \ I(r_{0})=E(r_{0}),\nonumber  
	\end{align}  
	which gives that 
	\begin{align}
	I(t)=\left\lbrace E^{-\beta}(r_0)- 2^{-(1+\beta)}\beta\int_{r_0}^{t} \min \{ g_{1,2}(s),g_{2,1}(s) \} ds\right\rbrace^{-\frac{1}{\beta}}. \nonumber
	\end{align} 
	For the above equation, the explosion time is given by
	\begin{align}
	\tau_{m_{i}}=\tau_{w_{i}} \leq \theta_{1} := \inf \left\lbrace t>r_{0} : \int_{r_{0}}^{t} \min \{ g_{1,2}(s),g_{2,1}(s) \} ds \geq \frac{2^{1+\beta}}{\beta \left( r_{1}e^{-\gamma_{1}r_{0}}+r_{2}e^{-\gamma_{2}r_{0}} \right)^{\beta}}    \right\rbrace, \nonumber
	\end{align}
    for $i=1,2,$ where $r_{1}$ and $r_{2}$ are given in \eqref{nn2}. If $\tau$ is the blow-up time of the solution $u=(u_1,u_2)^{\top}$ of the system \eqref{b1}, then by Lemma \ref{lem1}, we have 
    \begin{align}
    \tau &\leq \min \left\lbrace  (r_{0}+\tau_{m_1})(1+2^{\alpha}), (r_{0}+\tau_{m_2})(1+2^{\alpha}) \right\rbrace \leq (r_{0}+\theta_{1})(1+2^{\alpha}). \nonumber
    \end{align} 
	\textbf{Case 2:} Suppose that $\beta_{1}>\beta_{2}>0$ in \eqref{eq1}. Then, we have
	\begin{align}\label{sa2}
	\frac{dE(t)}{dt} \geq \min\left\lbrace g_{1,2}(t),g_{2,1}(t) \right\rbrace  \left[ h_{2}^{1+\beta_{1}}(t)+h_{1}^{1+\beta_{2}}(t)\right]  .
	\end{align}
	The Young inequality states that if $1<b<\infty$, $\delta>0$, and $a=\frac{b}{b-1}$, then 
	\begin{eqnarray} \label{na10}
	xy\leq \frac{\delta^{a}x^{a}}{a}+\frac{\delta^{-b}y^{b}}{b}, \ x,y\geq 0. 
	\end{eqnarray}
	By setting $b=\frac{1+\beta_{1}}{1+\beta_{2}}, \ y=h_{2}^{1+\beta_{2}}(t), \ x=\epsilon, \ \delta=\left( \frac{1+\beta_{1}}{1+\beta_{2}} \right)^{\frac{1+\beta_{2}}{1+\beta_{1}}},$ and using the fact that $\beta_{2}<\beta_{1}$ in (\ref{na10}), it follows that for any $\epsilon>0,$
	\begin{eqnarray} 
	h_{2}^{1+\beta_{1}}(t) \geq \epsilon h_{2}^{1+\beta_{2}}(t)-D_{1}\epsilon^{\frac{1+\beta_{1}}{\beta_{1}-\beta_{2}}},\nonumber
	\end{eqnarray}
	where $D_{1}=\left( \frac{\beta_{1}-\beta_{2}}{1+\beta_{1}} \right)\left(\frac{1+\beta_{1}}{1+\beta_{2}} \right)^{\frac{1+\beta_{2}}{\beta_{1}-\beta_{2}}}.$
	Since $\epsilon_0$ is the minimum of the quantities given in  \eqref{nca3},  in particular, we have $\epsilon_0\leq ( h_{2}(r_0)/D_{1}^{1/1+\beta_{2}})^{\beta_{1}-\beta_{2}},$ so that 
	\begin{eqnarray}
	\epsilon_{0} h_{2}^{1+\beta_{2}}(r_0)-D_{1}\epsilon_{0}^{\frac{1+\beta_{1}}{\beta_{1}-\beta_{2}}}\geq 0\nonumber.
	\end{eqnarray}
	From  \eqref{sa2},  we have
	\begin{align}\label{nca2}
	\frac{dE(t)}{dt}\geq \min\left\lbrace g_{1,2}(t),g_{2,1}(t) \right\rbrace \left[ h_{1}^{1+\beta_{2}}(t)+\epsilon_{0} h_{2}^{1+\beta_{2}}(t)-D_{1}\epsilon_{0}^{\frac{1+\beta_{1}}{\beta_{1}-\beta_{2}}} \right].
	\end{align}
	Utilizing the inequality  \eqref{cta1}  again with $n=1+\beta_{2}, \ a=1$ and $b=\epsilon_{0}^\frac{1}{1+\beta_{2}}\frac{h_{1}(t)}{h_{2}(t)}$, and by using  \eqref{nca3},  we see that 
	\begin{eqnarray}
	h_{1}^{1+\beta_{2}}(t) + \epsilon_{0} h_{2}^{1+\beta_{2}}(t) \geq 2^{-(1+\beta_{2})} \left( h_{1}(t)+\epsilon_{0}^\frac{1}{1+\beta_{2}}h_{2}(t) \right)^{1+\beta_{2}}\geq  2^{-(1+\beta_{2})}\epsilon_{0}E^{1+\beta_{2}}(t)\nonumber.
	\end{eqnarray}
	From  \eqref{nca2},  we deduce 
	\begin{align}
	\frac{dE(t)}{dt}&\geq \min\left\lbrace g_{1,2}(t),g_{2,1}(t) \right\rbrace \left[2^{-(1+\beta_{2})}\epsilon_{0}E^{1+\beta_{2}}(t)-D_{1}\epsilon_{0}^{\frac{1+\beta_{1}}{\beta_{1}-\beta_{2}}} \right].\nonumber
	\end{align}
    Note that the condition  \eqref{nca4}  gives $E(t)\geq E(r_0)>0,$ and so 
    \begin{align}
	\frac{dE(t)}{E^{1+\beta_{2}}(t)}&\geq \min\left\lbrace g_{1,2}(t),g_{2,1}(t) \right\rbrace  \left[\frac{\epsilon_{0}}{2^{1+\beta_{2}}}-\frac{D_{1}\epsilon_{0}^{\frac{1+\beta_{1}}{\beta_{1}-\beta_{2}}}}{E^{1+\beta_{2}}(r_0)} \right]dt.\nonumber
	\end{align}
	Thus, by comparison argument (see Theorem 1.3 of \cite{teschl} and Appendix), we get $E(t) \geq I(t)$ for all $t \geq 0$, where $I(t)$ solves the differential equation
	\begin{align}
	\frac{dI(t)}{I^{1+\beta_{2}}(t)}= \min\left\lbrace g_{1,2}(t),g_{2,1}(t) \right\rbrace  \left[\frac{\epsilon_{0}}{2^{1+\beta_{2}}}-\frac{D_{1}\epsilon_{0}^{\frac{1+\beta_{1}}{\beta_{1}-\beta_{2}}}}{I^{1+\beta_{2}}(r_0)} \right]dt.\nonumber	
	\end{align}
	Therefore,
	\begin{align}
	I(t)=\left\lbrace I^{-\beta_{2}}(r_{0})-\beta_{2}\left[ \frac{\epsilon_0}{2^{1+\beta_{2}}}-\frac{ \epsilon_{0}^{\frac{1+\beta_{1}}{\beta_{1}-\beta_{2}}}D_{1} }{I^{1+\beta_{2}}(r_0)} \right]\displaystyle\int_{r_0}^{t} \min\left\lbrace g_{1,2}(s),g_{2,1}(s) \right\rbrace ds \right\rbrace^{\frac{-1}{\beta_{2}}}. \nonumber 
	\end{align}
	For the above inequality, the explosion time is given by
	\begin{align}
	\tau_{m_{i}}&=\tau_{w_{i}} \leq \theta_{2} := \inf \Bigg\{  t> r_0 : \int_{r_0}^{t} \min\left\lbrace g_{1,2}(s),g_{2,1}(s) \right\rbrace ds \geq \nonumber\\ 
	& \qquad\left[\beta_{2} \left( r_{1}e^{-\gamma_{1}r_{0}}+r_{2}e^{-\gamma_{2}r_{0}}\right)^{\beta_{2}} \left(\frac{\epsilon_0}{2^{1+\beta_{2}}}-\frac{ \epsilon_{0}^{\frac{1+\beta_{1}}{\beta_{1}-\beta_{2}}}D_{1}}{\left(r_{1}e^{-\gamma_{1}r_{0}}+r_{2}e^{-\gamma_{2}r_{0}}\right)^{1+\beta_{2}}}\right)\right]^{-1} \Bigg\}, \nonumber
	\end{align}
    for $i=1,2,$ where $r_{1}$ and $r_{2}$ are given in \eqref{nn2}. If $\tau$ is the blow-up time of $u=(u_1,u_2)^{\top},$ then by Lemma \ref{lem1}, we have 
    \begin{align}
    \tau &\leq \min \left\lbrace  (r_{0}+\tau_{m_1})(1+2^{\alpha}), (r_{0}+\tau_{m_2})(1+2^{\alpha}) \right\rbrace \leq (r_{0}+\theta_{2})(1+2^{\alpha}), \nonumber
    \end{align} 
    which completes the proof.
	\end{proof}

\begin{Rem}
	Note that for any $$0<\epsilon_0\leq\min \Bigg\{ 1, \left( h_{2}(r_0)/D_{1}^{1/1+\beta_{2}} \right)^{\beta_{1}-\beta_{2}} ,\left(\frac{r_1e^{-\gamma_1r_0}+r_2e^{-\gamma_2r_0}}{2D_{1}^{1/1+\beta_{2}}}\right)^{\beta_1-\beta_2}\Bigg\},$$ the conditions \eqref{nca3} and \eqref{nca4} are satisfied. 
\end{Rem}
	Next, we extend the results obtained in Lemma 1,  \cite{doz2020} on one dimensional fractional Brownian motion to two dimensions.
	\begin{lemma}\label{l2} Let $0<H<1$ and let $\left( B_{1}^{H}(t),B_{2}^{H}(t) \right)_{t \geq 0} $ be a  two-dimensional fractional Brownian motion with Hurst parameter $H$ defined on a probability space $\left( \Omega, \mathscr{F}, \mathbb{P} \right).$ If $\nu>0,$ then $$\mathbb{P}\left( \int_{0}^{\infty} e^{ B^{H}_{1}(s)+B^{H}_{2}(s)-\nu s} ds < \infty\right)=1.$$ If $\nu<0,$ then $$\mathbb{P}\left( \int_{0}^{\infty} e^{ B^{H}_{1}(s)+B^{H}_{2}(s)-\nu s} ds = \infty\right)=1.$$  
	\end{lemma}
	\begin{proof}
	As a consequence of the law of iterated logarithm for fractional Brownian motion \cite{Orey}, we have
	\begin{align}
	\limsup_{t \rightarrow \infty} \frac{B_{i}^{H}(t)}{t^{H}\sqrt{2\log\log t}}= 1,  \  \mathbb{P}\text{-a.s.,}\  \nonumber 
	\end{align}
	for $i=1,2.$ This means that for some $t_{0}>e$ and for all $t \geq t_{0},$ we have  $$B^{H}_{i}(t) \leq 2 t^{H}\sqrt{2\log\log t},\ i=1,2.$$ Assume that $\nu>0.$ Since $H<1,$ there exists $t_{1}> t_{0}$ such that  for each $t \geq t_{1}$ $$2t^{H}\sqrt{2\log\log t} \leq \frac{\nu t}{4}.$$ Therefore, for all $t \geq t_{1},$ we have
	\begin{align}
	\int_{t}^{\infty} e^{ B^{H}_{1}(s)+B^{H}_{2}(s)-\nu s} ds < \int_{t}^{\infty} e^{ \frac{\nu s}{4}+\frac{\nu s}{4}-\nu s} ds =\int_{t}^{\infty} e^{ -\frac{\nu s}{2}} ds \leq \frac{2}{\nu}< \infty,  \  \mathbb{P}\text{-a.s.}  \nonumber
	\end{align}
	Since $B_{1}^{H}$ and $B_{2}^{H}$ have continuous paths, it follows that $$\int_{0}^{\infty} e^{ B^{H}_{1}(s)+B^{H}_{2}(s)-\nu s} ds < \infty,  \  \mathbb{P}\text{-a.s.} $$
	If $\nu <0,$ using the fact that $B_{i}^{H}(t)=-B_{i}^{H}(t)$ in distribution, for all $t > t_{0},$ we have $$B^{H}_{i}(t) \geq -2 t^{H}\sqrt{2\log\log t}, \  \mathbb{P}\text{-a.s.}.$$ Taking $t_{1}>t_{0}$ such that $-\frac{\nu t}{4} > 2t^{H}\sqrt{2\log\log t},$ for all $t>t_{1},$ we have
	\begin{align}
	\int_{t}^{\infty} e^{ B^{H}_{1}(s)+B^{H}_{2}(s)-\nu s} ds &>\int_{t_{1}}^{\infty} e^{ B^{H}_{1}(s)+B^{H}_{2}(s)-\nu s} ds > \int_{t_{1}}^{\infty} e^{ -2t^{H}\sqrt{2\log \log t}-\nu t} ds\nonumber\\ &>\int_{t_{1}}^{\infty} e^{ -\frac{\nu s}{2}} ds = \infty, \  \mathbb{P}\text{-a.s.}  \nonumber
	\end{align}
    which completes the proof. 
	\end{proof}
	In the following theorem, we find upper bounds for the probability of non-explosive positive solution $u=(u_1,u_2)^{\top}$ of the system \eqref{b1}.	
	\begin{theorem}\label{thm2}
	Let $1/2 <H<1$ and $\tau$ be the blow-up time of the system \eqref{b1} with non-negative initial values $f_{i},\ \rho_{i}>0$ in \eqref{a4} for $i=1,2,$  and  $\gamma_{1} \geq \gamma_{2}>0.$ If $k_{1,2}>0 $ (cf. \eqref{49}),  then any positive nontrivial solution of the system \eqref{b1} is  local $\mathbb{P}$-a.s., that is, $\mathbb{P}\left[ \tau< \infty \right]=1$ for all bounded measurable functions $f_{i}, \ i=1,2,$ which does not vanish. Further,
	\begin{enumerate}
	\item [ 1.] If $k_{1,2}<0$ and $\beta_{1}=\beta_{2}=\beta,$  there exists $\eta > r_{0}$ such that for all $t>5(\eta+r_{0}),$   the probability of non-explosion of the system \eqref{b1} is upper bounded by 
	$$ \mathbb{P} [\tau \geq t]\leq \mathbb{P}\left[ \int_{\eta}^{\frac{t}{1+2^{\alpha}}-r_{0}} \min \left\lbrace g_{1,2}(s),g_{2,1}(s) \right\rbrace ds < \frac{2^{(1+\beta)}}{\beta \left( r_{1}e^{-\gamma_{1}r_{0}}+r_{2}e^{-\gamma_{2}r_{0}} \right)^{\beta}} \right]. $$  
	\item [ 2.] If $k_{1,2}<0$ and $\beta_{1}>\beta_{2}>0,$ 
     there exists $\eta > r_{0}$ such that for all $t>5(\eta+r_{0}),$ the probability of non-explosion  of the system \eqref{b1} is upper bounded by 
	\begin{align}
    &\mathbb{P} [\tau \geq t] \\&\leq \mathbb{P}\Bigg\{\int_{\eta}^{\frac{t}{1+2^{\alpha}}-r_{0}} \min \left\lbrace g_{1,2}(s),g_{2,1}(s) \right\rbrace ds \nonumber\\ & \qquad< \left[\beta_{2} \left( r_{1}e^{-\gamma_{1}r_{0}}+r_{2}e^{-\gamma_{2}r_{0}}\right)^{\beta_{2}} \left(\frac{\epsilon_0}{2^{1+\beta_{2}}}-\frac{ \epsilon_{0}^{\frac{1+\beta_{1}}{\beta_{1}-\beta_{2}}}D_{1}}{\left(r_{1}e^{-\gamma_{1}r_{0}}+r_{2}e^{-\gamma_{2}r_{0}}\right)^{1+\beta_{2}}}\right)\right]^{-1} \Bigg\}, \nonumber 
	\end{align}
	\end{enumerate}
	where the functions $g_{1,2},\ g_{2,1}$ and $r_1,\ r_2$ are defined in \eqref{48} and \eqref{nn2} respectively.
	\end{theorem} 
	\begin{proof}
    For $\beta_{1} \geq \beta_{2}>0$ and  $k_{1,2}>0,$ we have $$\exp\{ sk_{1,2} \} \leq \exp\{ sk_{2,1} \}, \ \mbox{for all} \ s\geq 0.$$ We can choose $t_{1}>r_{0}$ for all $s>t_{1}$ such that $$s^{-d \beta_{1}/\alpha} \exp \left\lbrace  sk_{1,2} \right\rbrace >\exp\left\lbrace  \frac{sk_{1,2}}{2} \right\rbrace .$$ Therefore 
	\begin{align}
	& \int_{r_{0}}^{t} 2^{-d(1+\beta_{1})/\alpha} e^{sk_{1,2}+\rho_{1}B_{1}^{H}(s)+\rho_{2}B_{2}^{H}(s)}s^{-d \beta_{1}/\alpha} \wedge 2^{-d(1+\beta_{2})/\alpha}e^{sk_{2,1}+\rho_{1}B_{1}^{H}(s)+\rho_{2}B_{2}^{H}(s)}s^{-d \beta_{2}/\alpha} ds \nonumber\\ 
	&= 2^{-d(1+\beta_{1})/\alpha}\int_{r_{0}}^{t} e^{sk_{1,2}+\rho_{1}B_{1}^{H}(s)+\rho_{2}B_{2}^{H}(s)}s^{-d \beta_{1}/\alpha} ds \nonumber\\ 
	&> 2^{-d(1+\beta_{1})/\alpha}\int_{t_{1}}^{t} e^{sk_{1,2}/2+\rho_{1}B_{1}^{H}(s)+\rho_{2}B_{2}^{H}(s)} ds. \nonumber
	\end{align}
	By Lemma \ref{l2}, we have
	\begin{align}
	\int_{t_{1}}^{t} e^{sk_{1,2}/2+\rho_{1}B_{1}^{H}(s)+\rho_{2}B_{2}^{H}(s)} ds \rightarrow \infty \  \ \mbox{ as } \ t\rightarrow \infty, \ \mathbb{P}\text{-a.s. }\nonumber
	\end{align}
	Hence by Theorem \ref{thm1}, $\mathbb{P}\left[\tau<\infty \right]=1$
and therefore for any nontrivial, positive initial conditions, the solution of system \eqref{s1} exhibits finite-time blow-up $\mathbb{P}$-a.s.  
				
	\noindent\textbf{Case 1:} If $k_{1,2}<0,$ then we can choose $\eta>r_{0}$ such that for all $s>\eta,$ $$\ s^{-d \beta/\alpha} \exp\{ sk_{1,2}\}>\exp\{ 2sk_{1,2}\}.$$ Therefore, by Theorem \ref{thm1} we have for all $t>5(\eta+r_0),$ 
	\begin{align}
	\mathbb{P} [\tau \geq t] &\leq \mathbb{P}\left[\theta_{1}>\frac{t}{1+2^{\alpha}}-r_{0}\right] \nonumber\\
	&= \mathbb{P}\left[ \int_{r_{0}}^{\frac{t}{1+2^{\alpha}}-r_{0}} \min \left\lbrace g_{1,2}(s),g_{2,1}(s) \right\rbrace ds < \frac{2^{(1+\beta)}}{\beta \left( r_{1}e^{-\gamma_{1}r_{0}}+r_{2}e^{-\gamma_{2}r_{0}} \right)^{\beta}} \right] \nonumber\\
	& \leq \mathbb{P}\left[ \int_{\eta}^{\frac{t}{1+2^{\alpha}}-r_{0}} \min \left\lbrace g_{1,2}(s),g_{2,1}(s) \right\rbrace ds < \frac{2^{(1+\beta)}}{\beta \left( r_{1}e^{-\gamma_{1}r_{0}}+r_{2}e^{-\gamma_{2}r_{0}} \right)^{\beta}} \right], \nonumber
	\end{align} 
		where the functions $g_{1,2},\ g_{2,1}$ and $r_1,\ r_2$ are defined in \eqref{48} and \eqref{nn2} respectively.
		
    \vskip 0.2 cm
	\noindent

	\noindent \textbf{Case 2:} If $k_{1,2}<0,$ then we can choose $\eta>r_{0}$ such that for all $s>\eta,$  $$s^{-d \beta_{1}/\alpha} \exp\{ sk_{1,2}\}>\exp\{ 2sk_{1,2}\}.$$ Therefore by Theorem \ref{thm1} we have for all $t>5(\eta+r_0),$ 
	\begin{align}
	\mathbb{P} [\tau \geq t] &\leq \mathbb{P}\left[\theta_{2}>\frac{t}{1+2^{\alpha}}-r_{0}\right] \nonumber\\
	&= \mathbb{P}\Bigg[ \int_{r_{0}}^{\frac{t}{1+2^{\alpha}}-r_{0}} \min \left\lbrace g_{1,2}(s),g_{2,1}(s) \right\rbrace ds \nonumber\\ &\qquad < \left[\beta_{2} \left( r_{1}e^{-\gamma_{1}r_{0}}+r_{2}e^{-\gamma_{2}r_{0}}\right)^{\beta_{2}} \left(\frac{\epsilon_0}{2^{1+\beta_{2}}}-\frac{ \epsilon_{0}^{\frac{1+\beta_{1}}{\beta_{1}-\beta_{2}}}D_{1}}{\left(r_{1}e^{-\gamma_{1}r_{0}}+r_{2}e^{-\gamma_{2}r_{0}}\right)^{1+\beta_{2}}}\right)\right]^{-1} \Bigg] \nonumber\\
	& \leq \mathbb{P}\Bigg[ \int_{\eta}^{\frac{t}{1+2^{\alpha}}-r_{0}} \min \left\lbrace g_{1,2}(s),g_{2,1}(s) \right\rbrace ds \nonumber\\ &\qquad< \left[\beta_{2} \left( r_{1}e^{-\gamma_{1}r_{0}}+r_{2}e^{-\gamma_{2}r_{0}}\right)^{\beta_{2}} \left(\frac{\epsilon_0}{2^{1+\beta_{2}}}-\frac{ \epsilon_{0}^{\frac{1+\beta_{1}}{\beta_{1}-\beta_{2}}}D_{1}}{\left(r_{1}e^{-\gamma_{1}r_{0}}+r_{2}e^{-\gamma_{2}r_{0}}\right)^{1+\beta_{2}}}\right)\right]^{-1} \Bigg], \nonumber
	\end{align} 	 
	where $g_{1,2},\ g_{2,1}$ and $r_1,\ r_2$ are defined in \eqref{48} and \eqref{nn2} respectively.
	\end{proof}

	\section{The case $H=\frac{1}{2}$}\label{se12}
	In this section, we consider the case $H=\frac{1}{2}$ and $(B_1^{\frac{1}{2}}(\cdot),B_1^{\frac{1}{2}}(\cdot))=(W_1(\cdot),W_2(\cdot))$, where $W_1(\cdot)$ and $W_2(\cdot)$ are two independent standard Brownian motions. We are devoted to establish the upper bounds for the probability of non-explosive positive solution $u=(u_1,u_2)^{\top}$ of the system \eqref{re3}. By using random transformations $$v_{i}(t,x) = \exp\{-k_{i1}W_{1}(t)-k_{i2}W_{2}(t)\}u_{i}(t,x),\ i=1,2,$$ for $t \geq 0,\ x \in \mathbb{R}^{d}$,  the system $\eqref{b1}$ is transformed into the following  system of random PDEs:
	\begin{equation}\label{re3}
	\left\{
	\begin{aligned}
	\frac{\partial v_{i}(t,x)}{\partial t}&=\left( \Delta_{\alpha}+\gamma_{i}-\frac{k_{i1}^{2}+k_{i2}^{2}}{2} \right)v_{i}(t,x)+e^{-k_{i1}W_{1}(t)-k_{i2}W_{2}(t)}\left( e^{k_{j1}W_{1}(t)+k_{j2}W_{2}(t)}v_{j}(t,x) \right)^{1+\beta_{i}},\\
	v_{i}(0,x)&=f_{i}(x), \ \  x \in \mathbb{R}^{d},  
	\end{aligned}
	\right.
	\end{equation}
	for $i=1,2, \ j\in \left\{ 1,2 \right\} / \left\{i\right\}.$
	Let $\mathit{{\Gamma}}_{i}=\gamma_{i}-\frac{k_{i1}^{2}+k_{i2}^{2}}{2},$ for $i=1,2.$ The results obtained in previous sections are valid by replacing the constants $\gamma_{1}$ and $\gamma_{2}$ by $\mathit{\Gamma}_{1}$ and  $\mathit{\Gamma}_{2}$ respectively. For $\mathit{\Gamma}_{i} \neq 0, i=1,2,$ the blow-up results are quite similar to the case of $H>\frac{1}{2}.$
	\begin{Rem}
	For $H=1/2$, by replacing the constants $\gamma_{1}$ and $\gamma_{2}$ by $\mathit{\Gamma}_{1}$ and  $\mathit{\Gamma}_{2}$ respectively and proceeding in the same way as in the proof of Theorem \ref{thm1},  we obtain the upper bounds for the blow-up of the system \eqref{re3}. 
	\end{Rem}
	
	Next, we obtain upper bounds for the probability of non-explosive positive solution $u=(u_1,u_2)^{\top}$ of the system \eqref{re3}.
	\begin{theorem} \label{thm5.1}
	Let $\left( W_{1}(t),W_{2}(t) \right)_{t \geq 0}$ be a  two-dimensional Brownian motion and $u=(u_1,u_2)^{\top}$ be a mild solution of the system  $\eqref{b1}$ with bounded measurable functions $f_{i}\geq 0$ for $i=1,2,$  as initial data. If $\mathit{\Gamma}_{1} \geq \mathit{\Gamma}_{2}>0$ and $k_{1,2}>0,$ then any positive nontrivial solution of the system \eqref{b1} is $\mathbb{P}$-a.s. local, that is, $\mathbb{P}\left[ \tau< \infty \right]=1$ for all $f_{i} \geq 0, i=1,2,$ which does not vanish. 
		
		

	\begin{itemize}
	\item [ 1.] If $\beta_{1}=\beta_{2}=\beta, \ \mathit{\Gamma}_{1} \geq \mathit{\Gamma}_{2}>0$ and $k_{1,2}<0,$ there exists $\eta > r_{0}$ such that for all $t>5(\eta+r_{0}),$ the probability of the non-explosion time of the system \eqref{b1} is upper bounded by 
	$$ \mathbb{P} [\tau \geq t]\leq \mathbb{P}\left[ \int_{\eta}^{\frac{t}{1+2^{\alpha}}-r_{0}} \min \left\lbrace g_{1,2}(s),g_{2,1}(s) \right\rbrace ds < \frac{2^{(1+\beta)}}{\beta \left( r_{1}e^{-\gamma_{1}r_{0}}+r_{2}e^{-\gamma_{2}r_{0}} \right)^{\beta}} \right]. $$   
	In particular, 
	$$ \mathbb{P} [\tau = \infty]\leq \mathbb{P}\left[ \int_{\eta}^{\infty} \min \left\lbrace g_{1,2}(s),g_{2,1}(s) \right\rbrace ds < \frac{2^{(1+\beta)}}{\beta \left( r_{1}e^{-\gamma_{1}r_{0}}+r_{2}e^{-\gamma_{2}r_{0}} \right)^{\beta}} \right]. $$
	\item [ 2.] If $\beta_{1}>\beta_{2}>0, \ \mathit{\Gamma}_{1} \geq \mathit{\Gamma}_{2}>0$ and 
	$k_{1,2}<0,$ there exists $\eta > r_{0}$ such that for all $t>5(\eta+r_{0}),$  the probability of non-explosion time of the system \eqref{b1} is upper bounded by 
	\begin{align}
	\mathbb{P} [\tau \geq t] &\leq \mathbb{P}\Bigg\{\int_{\eta}^{\frac{t}{1+2^{\alpha}}-r_{0}} \min \left\lbrace g_{1,2}(s),g_{2,1}(s) \right\rbrace ds <\mathit{N} \Bigg\}. \nonumber 
	\end{align}
	In particular,
	\begin{align}
	\mathbb{P} [\tau = \infty] \leq \mathbb{P}\Bigg\{\int_{\eta}^{\infty} \min \left\lbrace g_{1,2}(s),g_{2,1}(s) \right\rbrace ds < \mathit{N} \Bigg\}, \nonumber 
	\end{align}
	where \begin{align}
	\mathit{N}=\left[\beta_{2} \left( r_{1}e^{-\gamma_{1}r_{0}}+r_{2}e^{-\gamma_{2}r_{0}}\right)^{\beta_{2}} \left(\frac{\epsilon_0}{2^{1+\beta_{2}}}-\frac{ \epsilon_{0}^{\frac{1+\beta_{1}}{\beta_{1}-\beta_{2}}}D_{1}}{\left(r_{1}e^{-\gamma_{1}r_{0}}+r_{2}e^{-\gamma_{2}r_{0}}\right)^{1+\beta_{2}}}\right)\right]^{-1}, \nonumber 
	\end{align}
$r_1,\ r_2$ are defined in Theorem \ref{thm1}  and $$g_{i,j}(s)=2^{\frac{-d(1+\beta_{i})}{\alpha}}e^{k_{i,j}s+\rho_{1}W_{1}(s)+\rho_{2}W_{2}(s)}s^{\frac{-d\beta_{i}}{\alpha}},$$ where $\rho_{1}, \rho_{2}$ are defined in \eqref{a4} and
	\begin{align}\label{kij} 
	k_{i,j}=-\mathit{\Gamma}_{i}+(1+\beta_{i})\mathit{\Gamma}_{j} \ \ \mbox{for} \ i=1,2, \ j\in \left\{ 1,2 \right\} / \left\{i\right\}.
	\end{align} 
	\item [3.]	
	\begin{itemize}
	\item [(a)] If $\beta_{1}=\beta_{2}=\beta$ and $\mathit{\Gamma}_{1}=\mathit{\Gamma}_{2}=0,$ then any non trivial solution $u=(u_1,u_2)^{\top}$ of \eqref{b1} exhibits finite-time blow-up in dimension $d \leq \frac{\alpha}{\beta}.$ 
	\item [(b)]Suppose $\beta_{1}> \beta_{2}>0$ and $\mathit{\Gamma}_{1}=\mathit{\Gamma}_{2}=0,$ then any non trivial solution $u=(u_1,u_2)^{\top}$ of the system \eqref{b1} exhibits finite-time blow-up in dimension $d \leq \frac{\alpha}{\beta_{1}}.$ 
	\end{itemize}
	\end{itemize}
	\end{theorem}
	\begin{proof}
	The proof of parts 1 and 2 in Theorem \ref{thm5.1} can be established by proceeding in the same way as in the proof of Theorem \ref{thm2}. An upper bound for $\mathbb{P}[\tau = \infty]$ in the case of $k_{1,2}<0$ (cf. \eqref{kij}), can be obtained in a similar way.  Also, for the case $k_{1,2}>0,$ (cf. \eqref{kij}), the solution $u=(u_1,u_2)^{\top}$ of \eqref{b1} with $H=\frac{1}{2}$ exhibits blow-up in finite-time $\mathbb{P}$-a.s.   	
		
	Now we prove part 3 of Theorem \ref{thm5.1}. Let us assume that $\mathit{\Gamma}_{i}=0$ for $i=1,2,$  so that  the system of random PDE  in \eqref{re3} becomes: 
    \begin{equation} \label{ran2}
	\left\{
	\begin{aligned}
	\frac{\partial v_{i}(t,x)}{\partial t}&= \Delta_{\alpha}v_{i}(t,x)+e^{-k_{i1}W_{1}(t)-k_{i2}W_{2}(t)}\left( e^{k_{j1}W_{1}(t)+k_{j2}W_{2}(t)}v_{j}(t,x) \right)^{1+\beta_{i}},\\
	v_{i}(0,x)&=f_{i}(x), \ \  x \in \mathbb{R}^{d},  
	\end{aligned}
	\right.
	\end{equation}
	for $i=1,2, \ j\in \left\{ 1,2 \right\} / \left\{i\right\}.$

	The following two Lemmas are useful for the proof of part 3 of Theorem \ref{thm5.1}.  
	\begin{lemma}[Lemma 16, \cite{doz2020}] \label{e1}
	Let $W=\left(W_{1}(t),W_{2}(t)\right)_{t\geq 0}$ be a standard two-dimensional Brownian motion defined on a filtered probability space $\left( \Omega, \mathscr{F}, (\mathscr{F}_{t})_{t \geq 0}, \mathbb{P} \right).$ For $T>0$, the process 
	\begin{align}\label{x1}
	X \equiv \left\lbrace X(t)=X_{1}(t)+X_{2}(t), \  X_{i}(t)=\exp\{-t/2\}W_{i}(e^{t}-1),i=1,2, \ 0 \leq t \leq T \right\rbrace  
	\end{align} 
	is a Brownian motion which is equivalent to $W=\left(W_{1}(t)+W_{2}(t)\right)_{t\geq 0},$ that is, there exists a probability measure $\mathbb{Q}$ on $(\Omega,\mathscr{F})$ having the same null set as $\mathbb{P}$ such that $\left(X(t), \mathcal{M}(t),\mathbb{Q} \right)$ is a Wiener process, where $\mathcal{M}(t)=\sigma\left\lbrace \sigma\left\lbrace X(s), \ s\leq t \right\rbrace \cup \mathcal{N} \right\rbrace, \ 0 \leq t \leq T,$ with $\mathcal{N}=\left\lbrace B\in \mathscr{F}: \ \mathbb{P}(B)=0 \right\rbrace.$    
    \end{lemma}
	\begin{proof}
	For each $s \geq u,$ let $h_{i}(s,u)=\exp\{ (u-s)/2\},$ for $i=1,2.$ Then
	\begin{align} \label{y1}
	Y \equiv \left\lbrace Y(t)=\int_{0}^{t} 2\left[ h_{1}(s,u)dW_{1}(u)+h_{2}(s,u)dW_{2}(u)\right], \ 0 \leq t \leq T \right\rbrace
	\end{align}
	is a centered Gaussian process with covariance function \begin{align} 
	\mathbb{E}[Y(t)Y(s)]&=\mathbb{E}\Bigg[ \left(\int_{0}^{t} 2\left[ h_{1}(s,u)dW_{1}(u)+h_{2}(s,u)dW_{2}(u)\right] \right)\nonumber\\	&\quad\times \left( \int_{0}^{s} 2\left[ h_{1}(s,u)dW_{1}(u)+h_{2}(s,u)dW_{2}(u)\right] \right)  \Bigg] \nonumber\\ 	
	&=e^{\frac{-t-s}{2}}\mathbb{E}\left[ \left( \int_{0}^{t} e^{u/2}\left(dW_{1}(u)+dW_{2}(u) \right) \right) \left( \int_{0}^{s} e^{u/2}\left(dW_{1}(u)+dW_{2}(u) \right)\right) 	\right]	\nonumber\\
	&=e^{\frac{-t-s}{2}}\mathbb{E}\left[ \left( \int_{0}^{s} e^{u/2}\left(dW_{1}(u)+dW_{2}(u) \right) \right)^{2}\right] \nonumber\\
	&\quad+\mathbb{E}\left[\left( \int_{s}^{t}e^{u/2}\left(dW_{1}(u)+dW_{2}(u) \right)\right) \left(  \int_{0}^{s} e^{u/2}\left(dW_{1}(u)+dW_{2}(u) \right)\right) 	\right]	\nonumber\\
	&=2e^{\frac{-t-s}{2}}\left( e^{s}-1\right)  \nonumber\\
	&=2\left[ \exp\left\lbrace -\frac{t}{2}+\frac{s}{2}\right\rbrace -\exp\left\lbrace  -\frac{t}{2}-\frac{s}{2}\right\rbrace \right], \ s\leq t,\nonumber
	\end{align} 
	which is the same as the covariance function of $X$, that is, the covariance function of $X$ is
	\begin{align}
	\mathbb{E}[X(t)X(s)]&=\mathbb{E}\left[\left(X_{1}(t)+X_{2}(t) \right) \left(X_{1}(s)+X_{2}(s) \right)  \right] \nonumber\\   
	&=\mathbb{E}\left[e^{-t/2}W_{1}(e^{t}-1)e^{-s/2}W_{1}(e^{s}-1)  \right] +\mathbb{E}\left[
	e^{-t/2}W_{2}(e^{t}-1)e^{-s/2}W_{2}(e^{s}-1)  \right]  \nonumber\\ 
	&=e^{\frac{-t-s}{2}}[e^{s}-1]+e^{\frac{-t-s}{2}}[e^{s}-1]\nonumber\\      
	&=2\left[ \exp\left\lbrace -\frac{t}{2}+\frac{s}{2}\right\rbrace -\exp\left\lbrace  -\frac{t}{2}-\frac{s}{2}\right\rbrace \right], \ s\leq t. \nonumber 
	\end{align}
	Therefore $X$ and $Y$ are equivalent. We have to prove that $Y$ and $W$ are equivalent. From \eqref{y1}, we have that for $0 \leq t \leq T,$
	\begin{align}
	Y(t)=W_{1}(t)+W_{2}(t)-\int_{0}^{t} \left( \int_{0}^{s} h_{1}(s,u)dW_{1}(u)+h_{2}(s,u)dW_{2}(u)\right) ds. \nonumber
	\end{align} 
	For $0 \leq t \leq T,$ let
	\begin{align}
	M(t)=&\exp \Bigg\{ \int_{0}^{t} \int_{0}^{s} \left[ h_{1}(s,u)dW_{1}(u)+h_{2}(s,u)dW_{2}(u)\right] \left( dW_{1}(s)+dW_{2}(s)\right)\nonumber\\
	&\quad-\frac{1}{2} \int_{0}^{t}2\left( \int_{0}^{s} h_{1}(s,u)dW_{1}(u)+h_{2}(s,u)dW_{2}(u) \right)^{2} ds\Bigg\}. \nonumber
	\end{align}
	Therefore the process $\left\lbrace M(t), \ t \geq 0 \right\rbrace$ is a  local martingale. Then by using the same argument as in Lemma 16, \cite{doz2020}, we have $\left\lbrace M(t), \ t \geq 0 \right\rbrace$ is a martingale and $Y$ and $W$ are equivalent. Hence  $\mathbb{P}$ and $\mathbb{Q}$ are equivalent probability measures.
	\end{proof}
	\begin{lemma}[Lemma 17, \cite{doz2020}]\label{la1}
	Let $\rho_{i}>0$ be given in \eqref{a4} for $i=1,2$ and $Y$ be  given by $$Y(t)=\rho_1\exp\{-t/2\}W_{1}(e^{t}-1)+\rho_2\exp\{-t/2\}W_{2}(e^{t}-1), \ 0 \leq t \leq T.$$ Then
	\begin{align}
	\int^{\ln(T+1)}_{\ln(r_{0}+1)} \exp \left\lbrace Y(t) \right\rbrace \mathbf{1}_{\{Y(t) \geq 0\}} dt \rightarrow \infty \ \mbox{ as } \ T \rightarrow \infty. \nonumber 
	\end{align} 
	\end{lemma}
    \noindent 
	\emph{Proof  of  Theorem  \ref{thm5.1}  (3).}	If $\beta_{1}=\beta_{2}=\beta,$  we obtain
	\begin{align}\label{5.3}
	&\int_{r_{0}}^{T} 2^{-d(1+\beta)/\alpha} e^{\rho_{1}W_{1}(s)+\rho_{2}W_{2}(s)}s^{-d \beta/\alpha} \wedge 2^{-d(1+\beta)/\alpha} e^{\rho_{1}W_{1}(s)+\rho_{2}W_{2}(s)}s^{-d \beta/\alpha} ds\nonumber \\
	&= 2^{-d(1+\beta)/\alpha}\int_{r_{0}}^{T}  e^{\rho_{1}W_{1}(s)+\rho_{2}W_{2}(s)}s^{-d \beta/\alpha} ds.
	\end{align}
	By using change of variables $s=e^{t}-1$ in \eqref{5.3}, we have 
	\begin{align}\label{5.4}
	\int_{\ln(r_{0}+1)}^{\ln(T+1)}  \frac{e^{\rho_{1}e^{t/2}X_{1}(t)+\rho_{2}e^{t/2}X_{2}(t)}}{(e^{t}-1)^{d \beta/\alpha}} e^{t} dt &= \int_{\ln(r_{0}+1)}^{\ln(T+1)}  \frac{e^{\rho_{1}e^{t/2}X_{1}(t)+\rho_{2}e^{t/2}X_{2}(t)+t}}{(e^{t}-1)^{d \beta/\alpha}}  dt \\
	&\geq \int_{\ln(r_{0}+1)}^{\ln(T+1)}  \frac{e^{\rho_{1}X_{1}(t)+\rho_{2}X_{2}(t)+t}}{e^{td \beta/\alpha}} \mathbf{1}_{\{Y(t)=\rho_{1}X_{1}(t)+\rho_{2}X_{2}(t)\geq 0\}} dt \nonumber\\
	&=\int_{\ln(r_{0}+1)}^{\ln(T+1)}  e^{Y(t)-t(d \beta/\alpha-1)} \mathbf{1}_{\{Y(t)\geq 0\}} dt \nonumber\\
	& \geq \int_{\ln(r_{0}+1)}^{\ln(T+1)}  e^{Y(t)} \mathbf{1}_{\{Y(t)\geq 0\}} dt ,\nonumber 
	\end{align} 
	provided $d \beta / \alpha \leq 1.$ 	By Lemma \ref{la1}, we have 
	\begin{align}
	\int_{\ln(r_{0}+1)}^{\ln(T+1)}  e^{Y(t)} \mathbf{1}_{\{Y(t)\geq 0\}} dt \rightarrow \infty \ \mbox{ as } \ T\rightarrow \infty, \  \mathbb{P}\mbox{-a.s.} \nonumber	
	\end{align} 
	Therefore, for each $N_{0}>0$  there exists $r_{0}>0$ such that  
	\begin{align}
	\mathit{A}_{T}^{N_{0}}=\left[ \int_{\ln(r_{0}+1)}^{\ln(T+1)}  e^{Y(t)} \mathbf{1}_{\{Y(t)\geq 0\}} dt>N_{0}\right], \nonumber 
	\end{align}   
for all $T>r_{0}$ and
	\begin{align}
	\mathbb{Q}\left[ \liminf_{T\rightarrow \infty} \left( \int_{\ln(r_{0}+1)}^{\ln(T+1)}  \frac{e^{\rho_{1}e^{t/2}X_{1}(t)+\rho_{2}e^{t/2}X_{2}(t)+t}}{(e^{t}-1)^{d \beta/\alpha}}  dt > N_{0}\right)  \right] \geq \mathbb{Q}\left[ \liminf_{T\rightarrow \infty} \left(\mathit{A}_{T}^{N_{0}}  \right)  \right]=1. \nonumber
	\end{align}
	By Lemma \ref{e1}, the probability measures $\mathbb{P}$ and $\mathbb{Q}$ are equivalent. Therefore 
	\begin{align}
	\mathbb{P}\left[ \liminf_{T\rightarrow \infty} \left( \int_{r_{0}}^{T} \min \left\lbrace g_{1,2}(s),g_{2,1}(s) \right\rbrace ds>N_{0} \right)  \right]=1, \nonumber
	\end{align} 
	and we conclude that for any nontrivial positive initial values, the system of random PDEs \eqref{ran2} exhibits a finite-time blow-up for dimension $d \leq \alpha/\beta.$ \\\\ \indent 
	If $\beta_{1}>\beta_{2},$ then we obtain
	\begin{align}\label{l5}
	&\int_{r_{0}}^{T}  2^{-d(1+\beta_{1})/\alpha}e^{\rho_{1}W_{1}(s)+\rho_{2}W_{2}(s)}s^{-d \beta_{1}/\alpha} \wedge 2^{-d(1+\beta_{2})/\alpha}e^{\rho_{1}W_{1}(s)+\rho_{2}W_{2}(s)}s^{-d \beta_{2}/\alpha} ds\nonumber \\
	&= 2^{-d(1+\beta_{1})/\alpha} \int_{r_{0}}^{T}  e^{\rho_{1}W_{1}(s)+\rho_{2}W_{2}(s)}s^{-d \beta_{1}/\alpha} ds.
	\end{align}
	Following similar   arguments as above,  for $N_{1}>0$, we obtain
	\begin{align}
	\mathbb{P}\left[ \liminf_{T\rightarrow \infty} \left( \int_{r_{0}}^{T} \min \left\lbrace g_{1,2}(s),g_{2,1}(s) \right\rbrace ds>N_{1} \right)  \right]=1, \nonumber
	\end{align} 
	so that	 for any nontrivial positive initial values, the system of random PDEs \eqref{ran2} exhibits a finite-time blow-up for dimension $d \leq \alpha/\beta_{1}.$
	\end{proof}	
	\begin{Rem}
	For the case $H=\frac{1}{2},\ \alpha=2,\ \beta_{1}=\beta_{2}=\beta$, the result of part 1 of Theorem 5.1 coincide with Theorem 4.2 of \cite{li}.
	\end{Rem}
	
    \section{Existence of global solutions} \label{sec5}
    In this section, we provide a sufficient condition for the existence of global mild solution of the system \eqref{b1} for $1/2 \leq H<1$ such that \eqref{a4} holds. For $i=1,2,$ let
    \begin{align}\label{mn1}
    \mathit{N}_{i}=\left\{ \begin{array}{rcl}
    \gamma_{i}, \ \ \ \ \ & H> 1/2,\\\\
    \gamma_{i}- \frac{k_{i1}^{2}+k_{i2}^{2}}{2},&  H=1/2.
    \end{array}\right.
    \end{align}
    
    \begin{theorem} \label{tt1}
    	Let $1/2 \leq H<1.$ If the initial values of the form $f_{1}=C_{1}\psi, f_{2}=C_{2}\psi,$	where $C_{1}$ and $C_{2}$ are any positive constants with $C_{1}\leq C_{2}$ and $\psi \in C_{c}^{\infty}({\mathbb R}^d)$ satisfy the inequalities 
    	\begin{align}\label{cd1}
    	\beta_{i} \int_{0}^{\infty}\exp\{\rho_{1} B^{H}_{1}(r)+\rho_{2} B^{H}_{2}(r)\}\left\| \exp\{ \mathit{N}_{i} r \}S_{r}f_{i} \right\|_{\infty}^{\beta_{i}}dr<1, \ i=1,2,
    	\end{align} 
    	then the mild solution  $u=(u_1,u_2)^{\top}$ of \eqref{b1} is global for all $(t,x) \in [0,\infty) \times \mathbb{R}^{d}.$ Moreover, for all $(t,x) \in [0,\infty) \times \mathbb{R}^{d}$
    	\begin{align*}
    	u_{i}(t,x) \leq \frac{\exp\{k_{i1}B^{H}_{1}(t)+k_{i2}B^{H}_{2}(t)+\mathit{N}_{i}t\}S_{t}f_{i}(x)}{\left(1-\beta_{i} \int_{0}^{\infty}\exp\{\rho_{1} B^{H}_{1}(r)+\rho_{2} B^{H}_{2}(r)\}\left\| \exp\{ \mathit{N}_{i}r \}S_{r}f_{i} \right\|_{\infty}^{\beta_{i}}dr \right)^{\frac{1}{\beta_{i}}}}, \ i=1,2. 
    	\end{align*} 
    	where $\rho_1,\ \rho_2$ are defined in \eqref{a4} and $\mathit{N}_{1},\mathit{N}_{2}$ are given  in \eqref{mn1}.
    \end{theorem}
    The proof of Theorem \ref{tt1} follows from Theorem \ref{t2} and Theorem 3.a in \cite{weisser}. \textcolor{blue}{The following provides an example which validates the condition \eqref{cd1} for the system \eqref{b1}.}    
    
    \textcolor{blue}{\begin{Ex} By properties (1) and (2) of Lemma \ref{l1}, we have
    \begin{align}\label{Bcd4}
   	\left\|S_{r}f_{i} \right\|_{\infty}^{\beta_{i}}&=\left( \sup_{x \in \mathbb{R}^{d}} \int_{\mathbb{R}^{d}}p(r,y-x)f_{i}(y)dy\right)^{\beta_{i}}\nonumber\\
   	&=r^{-d \beta_{i}/\alpha} \left( \sup_{x \in \mathbb{R}^{d}} \int_{\mathbb{R}^{d}}p(1,r^{-1/\alpha}(y-x))f_{i}(y)dy\right)^{\beta_{i}}\nonumber\\
   	& \leq r^{-d \beta_{i}/\alpha} p(1,0)^{\beta_{i}}\left\| f_{i} \right\|_{1}^{\beta_{i}},\ i=1,2. 
    \end{align}
    Using \eqref{Bcd4} in \eqref{cd1}, we have 
    \begin{align}\label{Bc1}
    	\beta_{i} &\int_{0}^{\infty}\exp\{\rho_{1} B^{H}_{1}(r)+\rho_{2} B^{H}_{2}(r)\}\left\| \exp\{ \mathit{N}_{i}r \}S_{r}f_{i} \right\|_{\infty}^{\beta_{i}}dr\nonumber\\
    	&\qquad\leq \beta_{i} p(1,0)^{\beta_{i}} \left\| f_{i} \right\|_{1}^{\beta_{i}}  \int_{0}^{\infty}\exp\{\rho_{1} B^{H}_{1}(r)+\rho_{2} B^{H}_{2}(r)+ \mathit{N}_{i} \beta_{i} r\} r^{-d \beta_{i} /\alpha} dr.
    \end{align} 
    If $\beta_{1}=\beta_{2}=1, k_{11}=k_{12}=k_{21}=k_{22}=1$, $\gamma_{1}=\gamma_{2}=\gamma$ and using \eqref{mn1}, then \eqref{Bc1} becomes
    \begin{align}\label{Bc2}
    	\beta_{i} &p(1,0)^{\beta_{i}} \left\| f_{i} \right\|_{1}^{\beta_{i}}  \int_{0}^{\infty}\exp\{\rho_{1} B^{H}_{1}(r)+\rho_{2} B^{H}_{2}(r)+ \mathit{N}_{i} \beta_{i} r\} r^{-d \beta_{i} /\alpha} dr\nonumber\\
    	& \leq p(1,0) \left\| f_{i} \right\|_{1}  \int_{0}^{\infty}\exp\{ B^{H}_{1}(r)+ B^{H}_{2}(r)+ \gamma r\} r^{-d /\alpha} dr.
    \end{align} 
    Let us denote 
    \begin{align}
    B^H_{\ast}(t_{1}) = \sup_{0 \leq s \leq t_{1}} |B^{H}_{1}(s)+B^{H}_{2}(s)|,\ \mbox{ for each }\ t_{1} \geq 0. \nonumber 
    \end{align}
    Therefore,
    \begin{align}
    \int_{0}^{t_{1}}e^{ B^{H}_{1}(r)+ B^{H}_{2}(r)+ \gamma r} r^{-d /\alpha} dr &\leq  e^{B^{H}_{\ast}(t_{1})}\int_{0}^{t_{1}}e^{\gamma r} r^{-d /\alpha} dr\leq  e^{ B^{H}_{\ast}(t_{1})}\int_{0}^{t_{1}}e^{2\gamma r} r^{-d /\alpha} dr.
    \end{align} 
    As a consequence of the law of iterated logarithm for fractional Brownian motions \cite{Orey}, we get
    \begin{align}
    \limsup_{t \rightarrow \infty} \frac{B_{i}^{H}(t)}{t^{H}\sqrt{2\log\log t}}= 1,  \  \mathbb{P}\text{-a.s.,}\ \mbox{ for } i=1,2.  \nonumber 
    \end{align}
     This means that for some $t_{0}>e$ and for all $t \geq t_{0},$ we have  $$B^{H}_{i}(t) \leq 2 t^{H}\sqrt{2\log\log t},\ i=1,2.$$ Assume that $\gamma>0.$ Since $H<1,$ there exists $t_{1}> t_{0}$ such that  for each $t \geq t_{1}$ $$2t^{H}\sqrt{2\log\log t} \leq \frac{\gamma t}{2}.$$ Therefore, for all $t \geq t_{1},$ we obtain
    \begin{align}
    \int_{t_{1}}^{\infty}e^{ B^{H}_{1}(r)+ B^{H}_{2}(r)+ \gamma r} r^{-d /\alpha} dr \leq 	\int_{t_{1}}^{\infty}e^{2\gamma r} r^{-d /\alpha} dr. 
    \end{align}
    Note that the improper integral $\int_{0}^{\infty}e^{ 2\gamma r} r^{-d /\alpha} dr$  converges, if $d/\alpha>1$ and let us take $C$ to be the value of the integral. Choosing  $f_{1}(x)=f_{2}(x)=\frac{\psi(x)}{2Cp(1,0)e^{ B^{H}_{\ast}(t_{1})}}, \ x \in \mathbb{R}^{d}$ with $\|\psi\|_{1}=1,$ we get 
    \begin{align*}
    p(1,0)& \left\| f_{i} \right\|_{1} \left(  \int_{0}^{t_{1}}e^{ B^{H}_{1}(r)+ B^{H}_{2}(r)+ \gamma r} r^{-d /\alpha} dr+\int_{t_{1}}^{\infty}e^{ B^{H}_{1}(r)+ B^{H}_{2}(r)+ \gamma r} r^{-d /\alpha} dr\right) \nonumber\\
    &\leq p(1,0) \frac{\left\| \psi \right\|_{1}}{2Cp(1,0)e^{ B^{H}_{\ast}(t_{1})}} \left(  e^{ B^{H}_{\ast}(t_{1})}\int_{0}^{t_{1}}e^{2\gamma r} r^{-d /\alpha} dr+\int_{t_{1}}^{\infty}e^{2\gamma r} r^{-d /\alpha} dr \right)\\
    &\leq  \frac{\left\| \psi \right\|_{1}}{2}<1,
    \end{align*}
    which verifies the condition \eqref{cd1}. 
    \end{Ex}}
    \begin{Rem}
    We have to provide the upper bounds for the probability that the solution $u=(u_1,u_2)^{\top}$ does not blow-up for $1/2 \leq H<1$ with the parameter \eqref{mn1}. First, we find the upper bounds by using the same argument as in Theorem \ref{thm1}.
    	
    	By using Theorem \ref{thm2}, if $\beta_{1}=\beta_{2}=\beta,$ then an upper bound of the probability that the solution $u=(u_1,u_2)^{\top}$ of the system \eqref{b1} does not blow-up  is given by   
    	\begin{align}
    	\mathbb{P} [\tau = +\infty]&=\mathbb{P} [\tau \geq t \ \mbox{ for all } \ t] \nonumber\\
    	&\leq \mathbb{P}\left[\theta_{1}>\frac{t}{1+2^{\alpha}}-r_{0} \ \ \mbox{ for all } \ t \right] \nonumber\\
    	&= \mathbb{P}\left[ \int_{r_{0}}^{\infty} \min \left\lbrace g_{1,2}(s),g_{2,1}(s) \right\rbrace s^{-d \beta/\alpha}ds < \frac{2^{(1+\beta)}}{\beta \left( r_{1}e^{-\gamma_{1}r_{0}}+r_{2}e^{-\gamma_{2}r_{0}} \right)^{\beta}} \right], \nonumber
    	\end{align}
    	where $g_{i,j}$ and $r_{i}$ are defined in \eqref{48} and \eqref{nn2} respectively, and one has to replace $\gamma_{i}$ by $\mathit{N}_{i}$ in $g_{i,j}$ for $i=1,2, \ j\in \left\{ 1,2 \right\} / \left\{i\right\}.$ 
    	
    	If $\beta_{1}>\beta_{2},$ then the probability that the solution $u=(u_1,u_2)^{\top}$ of \eqref{b1} does not blow-up is upper bounded by   
    	\begin{align}
    	\mathbb{P} [\tau = +\infty]&=\mathbb{P} [\tau \geq t \ \mbox{ for all } \ t] \nonumber\\
    	&\leq \mathbb{P}\left[\theta_{2}>\frac{t}{1+2^{\alpha}}-r_{0} \ \ \mbox{ for all } \ t \right] \nonumber\\
    	&=\mathbb{P}\Bigg\{ \int_{r_{0}}^{+\infty} \min \left\lbrace g_{1,2}(s),g_{2,1}(s) \right\rbrace ds \nonumber\\ & \qquad< \left[\beta_{2} \left( r_{1}e^{-\gamma_{1}r_{0}}+r_{2}e^{-\gamma_{2}r_{0}}\right)^{\beta_{2}} \left(\frac{\epsilon_0}{2^{1+\beta_{2}}}-\frac{ \epsilon_{0}^{\frac{1+\beta_{1}}{\beta_{1}-\beta_{2}}}D_{1}}{\left(r_{1}e^{-\gamma_{1}r_{0}}+r_{2}e^{-\gamma_{2}r_{0}}\right)^{1+\beta_{2}}}\right)\right]^{-1} \Bigg\}, \nonumber
    	\end{align}
    		where $g_{i,j}$ and $r_{i}$ are defined in \eqref{48} and \eqref{nn2} respectively, and  $\gamma_{i}$ in $g_{i,j}$ has to be replaced by $\mathit{N}_{i}$  for $i=1,2, \ j\in \left\{ 1,2 \right\} / \left\{i\right\}.$

    	Next, we find the lower bound for the probability that the solution $u=(u_1,u_2)^{\top}$ does not blow-up. 
        By using \eqref{Bcd4}, we get
    	\begin{align}
    	&\int_{0}^{t} \exp\{\rho_{1} B^{H}_{1}(r)+\rho_{2}B^{H}_{2}(r)+\mathit{N}_{i}\beta_{1}r\}\left\|S_{r}f_{i} \right\|_{\infty}^{\beta_{i}}dr \nonumber\\
    	&\qquad \qquad \leq p(1,0)^{\beta_{i}}\left\| f_{i} \right\|_{1}^{\beta_{i}}\int_{0}^{t} \exp\{\rho_{1} B^{H}_{1}(r)+\rho_{2}B^{H}_{2}(r)+\mathit{N}_{i}\beta_{1}r\}r^{-d \beta_{i}/\alpha} dr.  \nonumber
    	\end{align}
    	By Theorem \ref{t2}, we have 
    	\begin{align}
    	\mathbb{P} [\tau_{\ast} = +\infty] &= \mathbb{P}\left[ \int_{0}^{t} \exp\{\rho_{1} B^{H}_{1}(r)+\rho_{2}B^{H}_{2}(r)+\mathit{N}_{i}\beta_{i}r\}\left\|S_{r}f_{i} \right\|_{\infty}^{\beta_{i}}dr 
    	< \min \Bigg\{ \frac{1}{\beta_{1}},  \frac{1}{\beta_{2}} \Bigg\} \right] \nonumber\\
    	&\geq \mathbb{P}\Bigg[\int_{0}^{t} \exp\{\rho_{1} B^{H}_{1}(r)+\rho_{2}B^{H}_{2}(r)+\mathit{N}_{i}r\}r^{\frac{-d \beta_{i}}{\alpha}} dr \nonumber\\ 
    	&\qquad \qquad < \min \Bigg\{ \frac{1}{p(1,0)^{\beta_{1}}\left\| f_{1} \right\|_{1}^{\beta_{1}}\beta_{1}},  \frac{1}{p(1,0)^{\beta_{2}}\left\| f_{2} \right\|_{1}^{\beta_{2}}\beta_{2}} \Bigg\} \Bigg]. \nonumber
    	\end{align} 
    \end{Rem} 
    
    \begin{Rem}
    	For the case $H=\frac{1}{2},\ \rho_{i}>0\ \mbox{and}\  \gamma_{i}=\frac{k_{i1}^{2}+k_{i2}^{2}}{2},\ i=1,2$, we have $\mathit{N}_{i}=0,\ i=1,2.$ Sufficient condition for the global existence of the solution $u=(u_1,u_2)^{\top}$ of the system \eqref{b1} is given by Theorem \ref{tt1}. Under the assumptions 
    	\begin{align}\label{cd6}
    	d > \alpha/\beta_{2}\ \mbox{and}\ 0 \leq f_{i}(x) \leq C_{i}p(\zeta,x) \leq C_{2} p(\zeta,x),\ i=1,2, \ x \in \mathbb{R}^{d},
    	\end{align} 
    	where $\zeta>0, C_{1}$ and $C_{2}$ are suitably chosen constants with $C_{1} \leq C_{2}$, conditions \eqref{cd1} show that the improper integrals $ \displaystyle\int_{\varepsilon}^{\infty} \frac{e^{\rho_{1}W_{1}(s)+\rho_{2}W_{2}(s)}}{s^{d \beta_{i}/\alpha}} ds$  converge for $i=1,2$ and for some $\varepsilon >0.$ By Lemma \ref{l1}, we have
    	\begin{align}
    	S_{r}f_{i}(x)&= \int_{\mathbb{R}^{d}} p(r,x-y)f_{i}(y)dy \leq C_{2} \int_{\mathbb{R}^{d}} p(r,x-y)p(\zeta,y)dy= C_{2} p(r+\zeta,x) \nonumber \\
    	&=C_{2}(r+\zeta)^{-d/\alpha}p(1,(r+\zeta)^{-1/\alpha}x).\nonumber 
    	\end{align} 
    	Therefore,
    	\begin{align}
    	\int_{0}^{\infty}e^{\rho_{1} W_{1}(r)+\rho_{2} W_{2}(r)}\left\|S_{r}f_{i} \right\|_{\infty}^{\beta_{i}}dr&\leq C_{2}^{\beta_{i}} p(1,0)^{\beta_{i}}\int_{0}^{\infty}e^{\rho_{1} W_{1}(r)+\rho_{2} W_{2}(r)}(r+\zeta)^{-d \beta_{i}/\alpha}dr \nonumber\\
    	&=C_{2}^{\beta_{i}} p(1,0)^{\beta_{i}}\int_{0}^{\infty}e^{\rho_{1} W_{1}(e^{t}-1)+\rho_{2} W_{2}(e^{t}-1)+t}(e^{t}-1+\zeta)^{-d \beta_{i}/\alpha}dt, \nonumber 
    	\end{align}
    	for $i=1,2.$ In view of Lemma \ref{e1}, we need to estimate the finiteness of the integral $\int_{\varepsilon}^{\infty}\exp\{\rho_{1}e^{t/2}X_{1}(t)+\rho_{2}e^{t/2}X_{2}(t)-b_{i} t\} dt $ for any $\varepsilon >0$ and $b_{i}=d \beta_{i}/\alpha-1, i=1,2.$ However, it is shown in Lemma \ref{la3} below that this integral  diverges for any $\varepsilon>0$, even if $d>\alpha/\beta_{2}.$ This shows that the stochastic perturbation in the system \eqref{b1} qualitatively affects the behavior of solution. The  reason is that $\Delta_{\alpha}$ has not enough dissipativity. Dissipativity can be increased by replacing $\Delta_{\alpha}$ by a non autonomous differential operator of the type $e^{t}\Delta_{\alpha},$ which transforms \eqref{cd1} into
    	\begin{align}\label{cd5}
    	\int_{s}^{\infty}\exp\{\rho_{1} B^{H}_{1}(r)+\rho_{2} B^{H}_{2}(r)\}\left\| \exp\{ \mathit{N}_{i} r \}\mathcal{U}(r,s)f_{i} \right\|_{\infty}^{\beta_{i}}dr<1, \ i=1,2,
    	\end{align}
    	for some $s \geq 0,$ where $\left\lbrace \mathcal{U}(t,s);\ 0 \leq s \leq t\right\rbrace$ is the two parameter semigroup associated with this operator and is given by $\mathcal{U}(t,s)=S_{e^{t}-e^{s}}.$ By  using \eqref{Bcd4}, a simple calculation shows that \eqref{cd5} can be satisfied trajectory wise for  $\mathit{N}_{i}=0, i=1,2$, $H=1/2$ and for a  suitable choice of initial data $f_{i}, i=1,2$ as in \eqref{cd6}. For such a choice, the solution exist globally due to Lemma \ref{l2}.    	
    \end{Rem}    
    \begin{lemma}\label{la3}
    	If $\rho_{1}>0$ and $\rho_{2}>0,$ then the integral $\int_{\varepsilon}^{\infty}\exp\{\rho_{1}e^{t/2}X_{1}(t)+\rho_{2}e^{t/2}X_{2}(t)-b_{i}t\} dt ,$ where $b_{i}= \frac{d\beta_{i}}{\alpha}-1, i=1,2$ diverges for any $\varepsilon>0.$
    \end{lemma}
    \begin{proof}
    	Suppose that $\int_{\epsilon}^{\infty}\exp\{\rho_{1}e^{t/2}X_{1}(t)+\rho_{2}e^{t/2}X_{2}(t)-b_{i}t\} dt <\infty.$ This implies that $\rho_{1}e^{t/2}X_{1}(t)+\rho_{2}e^{t/2}X_{2}(t) \leq td \beta_{i}/\alpha$ for all $t$ sufficiently large. Since $\displaystyle e^{t/2}/t \rightarrow \infty$ as $t \rightarrow \infty,$ this would imply that $\rho_1 X_{1}(t)+\rho_2X_2(t) \rightarrow 0$ as $t \rightarrow \infty,$ which is not possible. Therefore $\rho_{1}e^{t/2}X_{1}(t)+\rho_{2}e^{t/2}X_{2}(t) > td \beta_{i}/\alpha,$ and hence  $\rho_{1}e^{t/2}X_{1}(t)+\rho_{2}e^{t/2}X_{2}(t) - b_{i}t>t.$ Therefore, the integral diverges.
    \end{proof}
    
   \section{A More General Case}\label{sec6}
    In this section, we consider the system \eqref{b1} with  parameters $1/2 \leq H<1$ and \eqref{mn1}. We find the lower and upper bounds to the system \eqref{b1} without the conditions given in \eqref{a4}.  From \eqref{e3}, we have
    \begin{align}
    v_{i}(t,x)&=S_{t}f_{i}(x)\nonumber\\
    &\quad+\int_{0}^{t} S_{t-r}\left[\mathit{N}_{i}v_{i}(s,\cdot)+ e^{-k_{i1}B^{H}_{1}(r)-k_{i2}B^{H}_{2}(r)}\left(e^{k_{j1}B^{H}_{1}(r)+k_{j2}B^{H}_{2}(r)}v_{j}(r,\cdot)\right)^{1+\beta_{i}} \right](x)dr. \nonumber
    \end{align}
    for $i=1,2,\ \left\lbrace j \right\rbrace=\left\lbrace1,2\right\rbrace/\{i\}$ and $\mathit{N}_{i}$ are defined in \eqref{mn1}. 
    
    Next result gives the lower bounds for finite-time blow-up of the solution $u=(u_1,u_2)^{\top}$ of the system \eqref{b1} when  $\gamma_{1}=\gamma_{2}=\lambda$. Hence $T_{1}(t)=T_{2}(t)=T(t)\ (say)$ which are defined in \eqref{op1}.
    \begin{theorem}\label{thm6.1}
    If $\beta_{1}, \beta_{2}$ are positive constants and the initial values are of the form $f_{1}=C_{1}\psi, f_{2}=C_{2}\psi$	where $C_{1}$ and $C_{2}$ are any positive constants with $C_{1}\leq C_{2}$ and $\psi \in C_{c}^{\infty}({\mathbb R}^d).$ Then $\tau_{\ast} \leq \tau,$ where $\tau_{\ast}$ is given by 
    	\begin{align}\label{7.1}
    	\tau_{\ast} = \inf \Bigg\{ t\geq 0 : &\int_{0}^{t}   e^{(1+\beta_{1})k_{21}-k_{11} B^{H}_{1}(r)+(1+\beta_{1})k_{22}-k_{12}B^{H}_{2}(r)+\lambda \beta_{1}r}r^{-d \beta_{1}/\alpha}dr \geq \frac{1}{\beta_{1}p(1,0)^{\beta_{1}}C_{1}^{\beta_{1}}\left\|\psi \right\|_{\infty}^{\beta_{1}}}, \nonumber\\ \mbox{or} \
    	&\int_{0}^{t}   e^{(1+\beta_{2})k_{11}-k_{21} B^{H}_{1}(r)+(1+\beta_{2})k_{12}-k_{22}B^{H}_{2}(r)+\lambda \beta_{2}r}r^{-d \beta_{2}/\alpha}dr \geq \frac{1}{\beta_{1}p(1,0)^{\beta_{1}}C_{1}^{\beta_{1}}\left\|\psi \right\|_{\infty}^{\beta_{1}}} \Bigg\}. 
    	\end{align}
    \end{theorem}
    \begin{proof}
    	For all $x\in \mathbb{R}^{d},  \ t\geq 0,$ let us define the operators $\mathcal{J}_{1}, \ \mathcal{J}_{2}$ as
    	\begin{align}
    	\mathcal{J}_{1}v(t,x)&=T(t)f_{1}(x) \nonumber\\&\quad+\int_{0}^{t} \exp\{(1+\beta_{1})k_{21}-k_{11} B^{H}_{1}(r)+(1+\beta_{1})k_{22}-k_{12}B^{H}_{2}(r)\}\left(T(t-r)v\right)^{1+\beta_{1}}dr,\nonumber\\
    	\mathcal{J}_{2} v(t,x)&=T(t)f_{2}(x)\nonumber\\&\quad+\int_{0}^{t}\exp\{(1+\beta_{2})k_{11}-k_{21} B^{H}_{1}(r)+(1+\beta_{2})k_{12}-k_{22}B^{H}_{2}(r)\}\left(T(t-r)v\right)^{1+\beta_{2}}dr,\nonumber
    	\end{align}
    	where $v$ is any non-negative, bounded and measurable function. Moreover, on the set $t < \tau_{*},$ we set
    	\begin{align}
    	\mathscr{G}_{1}(t)&=\left[ 1-\beta_{1} \int_{0}^{t}  \exp\{(1+\beta_{1})k_{21}-k_{11} B^{H}_{1}(r)+(1+\beta_{1})k_{22}-k_{12}B^{H}_{2}(r)\}\left\| T(r)f_{1} \right\|_{\infty}^{\beta_{1}}dr\right]^{\frac{-1}{\beta_{1}}},\nonumber\\
    	\mathscr{G}_{2}(t)&=\left[ 1- \beta_{2} \int_{0}^{t}  \exp\{(1+\beta_{2})k_{11}-k_{21} B^{H}_{1}(r)+(1+\beta_{2})k_{12}-k_{22}B^{H}_{2}(r)\}\left\| T(r)f_{2} \right\|_{\infty}^{\beta_{2}}dr\right]^{\frac{-1}{\beta_{2}}}.\nonumber
    	\end{align}
    	By using same way as in the proof of Theorem \ref{t2}, we get 
    	\begin{eqnarray}
    	v_{1}(t,x)=\mathcal{J}_{1} v_{2}(t,x), \  v_{2}(t,x)=\mathcal{J}_{2} v_{1}(t,x), \ \ x\in \mathbb{R}^{d}, \ \ 0 \leq t<\tau_{*}.\nonumber
    	\end{eqnarray}
    	Moreover,  we have
    	\begin{eqnarray} 
    	\mathcal{J}_{1} v (t,x)\leq T(t)f_{1}(x) \mathscr{G}_{1}(t)\ \text{ and }\ \mathcal{J}_{2} u(t,x)\leq T(t)f_{2}(x) \mathscr{G}_{2}(t), \nonumber
    	\end{eqnarray}
    	so that
    	\begin{align}
    	v_{1}(t,x) &\leq \frac{T(t)f_{1}(x)}{\left[ 1-\beta_{1}p(1,0)^{\beta_{1}}\left\| f_{1} \right\|_{\infty}^{\beta_{1}}\int_{0}^{t}  e^{(1+\beta_{1})k_{21}-k_{11} B^{H}_{1}(r)+(1+\beta_{1})k_{22}-k_{12}B^{H}_{2}(r)+\lambda \beta_{1}r}r^{-d \beta_{1}/\alpha}dr\right]^{\frac{1}{\beta_{1}}}}, \nonumber\\
    	v_{2}(t,x) &\leq \frac{T(t)f_{2}(x)}{\left[ 1-\beta_{2}p(1,0)^{\beta_{2}}\left\| f_{1} \right\|_{\infty}^{\beta_{2}}\int_{0}^{t}  e^{(1+\beta_{2})k_{11}-k_{21} B^{H}_{1}(r)+(1+\beta_{2})k_{12}-k_{22}B^{H}_{2}(r)+\lambda \beta_{2}r}r^{-d \beta_{2}/\alpha}dr\right]^{\frac{1}{\beta_{2}}}}, \nonumber	
    	\end{align}
    and one can complete the proof of \eqref{7.1}. 
    \end{proof}
    The following theorem gives an upper bounds for the finite-time blow-up  of solution $u=(u_1,u_2)^{\top}$ of the system \eqref{b1} for $1/2 \leq H <1$ with $\beta_{1} \geq \beta_{2}>0.$
    \begin{theorem} \label{thm6.2}
    Let $\frac{1}{2} \leq H<1$ and let $u=(u_1,u_2)^{\top}$ be a weak solution of \eqref{b1} with nonnegative initial data $f_1$ and $f_2$ which are  bounded measurable functions. 	
    	\item [1.] If $\beta_{1}=\beta_{2}=\beta>0.$ Then $u=(u_1,u_2)^{\top}$ blow-up in finite-time which is  given by
    	\begin{align}
    	\theta_{1} = \inf \Bigg\{ &t>r_{0} : \int_{r_{0}}^{t} e^{k_{1,2}s+((1+\beta)k_{21}-k_{11}) B^{H}_{1}(s)+((1+\beta)k_{22}-k_{12})B^{H}_{2}(s)}s^{\frac{-d\beta}{\alpha}}\nonumber\\ &\wedge e^{k_{2,1}s+((1+\beta)k_{21}-k_{11}) B^{H}_{1}(s)+((1+\beta)k_{22}-k_{12})B^{H}_{2}(s)}s^{\frac{-d\beta}{\alpha}} ds \geq \frac{2^{(1+\beta)}}{\beta \left(r_{1}e^{-\gamma_{1}r_{0}}+r_{2}e^{-\gamma_{2}r_{0}} \right)^{\beta}} \nonumber\\ \mbox{(or)} \
    	&\int_{r_{0}}^{t} e^{k_{1,2}s+((1+\beta)k_{11}-k_{21}) B^{H}_{1}(s)+((1+\beta)k_{12}-k_{22})B^{H}_{2}(s)}s^{\frac{-d\beta}{\alpha}}\nonumber\\ &\wedge e^{k_{2,1}s+((1+\beta)k_{11}-k_{21}) B^{H}_{1}(s)+((1+\beta)k_{12}-k_{22})B^{H}_{2}(s)}s^{\frac{-d\beta}{\alpha}} ds \geq \frac{2^{(1+\beta)}}{\beta \left(r_{1}e^{-\gamma_{1}r_{0}}+r_{2}e^{-\gamma_{2}r_{0}} \right)^{\beta}}    \Bigg\}. \nonumber 
    	\end{align}
    	Moreover, if $\tau$ is the blow-up time of \eqref{b1}, then 
    	\begin{align}
    	\tau &\leq \min \left\lbrace  (r_{0}+\tau_{m_1})(1+2^{\alpha}), (r_{0}+\tau_{m_2})(1+2^{\alpha}) \right\rbrace \leq (r_{0}+\theta_{1})(1+2^{\alpha}),\nonumber
    	\end{align} 
    	where $r_{i}=p(1,0)2^{-2d}\exp\{\gamma_{i}r_{0}\}\mathbb{E}\left[ f_{i}(X_{2^{-\alpha}r_{0}}) \right],$ and $k_{i,j}=-\mathit{N}_{i}+(1+\beta_{i})\mathit{N}_{j}$ for $i=1,2, \ j\in \left\{ 1,2 \right\} / \left\{i\right\}.$  
    	\item [2.] If $\beta_{1}>\beta_{2}>0$ and let $D_{1}=\left( \frac{\beta_{1}-\beta_{2}}{1+\beta_{1}} \right)\left(\frac{1+\beta_{1}}{1+\beta_{2}} \right)^{\frac{1+\beta_{2}}{\beta_{1}-\beta_{2}}},$
    	\begin{align}
    	\epsilon_{0}&\leq \min \Bigg\{ 1, \left( h_{1}(r_0)/D_{1}^{1/1+\beta_{2}} \right)^{\beta_{1}-\beta_{2}} \Bigg\}. \nonumber 
    	\end{align} 
    	Assume that 
    	\begin{align}
    	2^{-(1+\beta_{2})}\epsilon_{0}\left(r_{1}e^{-\gamma_{1}r_{0}}+r_{2}e^{-\gamma_{2}r_{0}}\right)^{1+\beta_{2}}\geq \epsilon_{0}^{\frac{1+\beta_{1}}{\beta_{1}-\beta_{2}}}D_{1}. \nonumber 
    	\end{align} 
    	Then $u=(u_1,u_2)^{\top}$ blow-up in finite-time which is given by
    	\begin{align}
    	\theta_{2} = \inf \Bigg\{  &t> r_0 : \int_{r_0}^{t}2^{-d(1+\beta_{1})/\alpha} e^{k_{1,2}s+((1+\beta_{1})k_{21}-k_{11}) B^{H}_{1}(s)+((1+\beta_{2})k_{22}-k_{12})B^{H}_{2}(s)}s^{\frac{-d\beta_{1}}{\alpha}}\nonumber\\ &\wedge 2^{-d(1+\beta_{2})/\alpha}e^{k_{2,1}s+((1+\beta_{1})k_{21}-k_{11}) B^{H}_{1}(s)+((1+\beta_{2})k_{22}-k_{12})B^{H}_{2}(s)}s^{\frac{-d\beta_{2}}{\alpha}} ds  \nonumber\\ 
    	& \qquad\geq\left[\beta_{2} \left(r_{1} e^{-\gamma_{1}r_{0}}+r_{2}e^{-\gamma_{2}r_{0}}\right)^{\beta_{2}} \left(\frac{\epsilon_0}{2^{1+\beta_{2}}}-\frac{ \epsilon_{0}^{\frac{1+\beta_{1}}{\beta_{1}-\beta_{2}}}D_{1}}{\left(r_{1}e^{-\gamma_{1}r_{0}}+r_{2}e^{-\gamma_{2}r_{0}}\right)^{1+\beta_{2}}}\right)\right]^{-1} \nonumber\\ \mbox{or} &\quad \int_{r_0}^{t}2^{-d(1+\beta_{1})/\alpha} e^{k_{1,2}s+((1+\beta_{2})k_{11}-k_{21}) B^{H}_{1}(s)+((1+\beta_{2})k_{12}-k_{22})B^{H}_{2}(s)}s^{\frac{-d\beta_{1}}{\alpha}}\nonumber\\ &\wedge 2^{-d(1+\beta_{2})/\alpha} e^{k_{2,1}s+((1+\beta_{2})k_{11}-k_{21}) B^{H}_{1}(s)+((1+\beta_{2})k_{12}-k_{22})B^{H}_{2}(s)}s^{\frac{-d\beta_{2}}{\alpha}} ds\nonumber\\ 
    	& \qquad \geq \left[\beta_{2} \left(r_{1} e^{-\gamma_{1}r_{0}}+r_{2}e^{-\gamma_{2}r_{0}}\right)^{\beta_{2}} \left(\frac{\epsilon_0}{2^{1+\beta_{2}}}-\frac{ \epsilon_{0}^{\frac{1+\beta_{1}}{\beta_{1}-\beta_{2}}}D_{1}}{\left(r_{1}e^{-\gamma_{1}r_{0}}+r_{2}e^{-\gamma_{2}r_{0}}\right)^{1+\beta_{2}}}\right)\right]^{-1} \Bigg\}. \nonumber
    	\end{align}
    	Moreover, if $\tau$ is the blow-up time of \eqref{b1}, then  
    	\begin{align}
    	\tau &\leq \min \left\lbrace  (r_{0}+\tau_{m_1})(1+2^{\alpha}), (r_{0}+\tau_{m_2})(1+2^{\alpha}) \right\rbrace \leq (r_{0}+\theta_{2})(1+2^{\alpha}), \nonumber
    	\end{align} 
    	where $r_{i}=p(1,0)2^{-2d}\exp\{\gamma_{i}r_{0}\}\mathbb{E}\left[ f_{i}(X_{2^{-\alpha}r_{0}}) \right],$ and $k_{i,j}=-\mathit{N}_{i}+(1+\beta_{i})\mathit{N}_{j}$ for $i=1,2, \ j\in \left\{ 1,2 \right\} / \left\{i\right\}.$ 
    \end{theorem}
    Proof of  Theorem \ref{thm6.2} can be obtained by proceeding in the same way as in the proof of Theorem \ref{thm1}.
	 
  \section{Appendix}
Here,  we provide  a  comparison result  which is used in Section \ref{sec4} (see Theorem 1.3 of \cite{teschl} also). 
  \begin{align}
  	w_{i}(t) &\geq r_{i}e^{-\gamma_{i}r_{0}}+\int_{r_{0}}^{t}g_{i,j}(r)w_{j}^{1+\beta_{i}}(r) dr, \ \ t \geq r_{0}, \nonumber
  \end{align}
  and
  \begin{align*}
  	h_{i}(t)=r_{i}e^{-\gamma_{i}r_{0}}+\int_{r_{0}}^{t}g_{i,j}(r)h_{j}^{1+\beta_{i}}(r) dr,\ i=1,2,\ \left\lbrace j \right\rbrace=\left\lbrace1,2\right\rbrace/\{i\}. 
  \end{align*}
  We have to show that $w_{i}(t) \geq h_{i}(t),\ i=1,2,$ for $t \geq r_{0}.$	Let $$	x_{i}(t)=r_{i}e^{-\gamma_{i}r_{0}}+\int_{r_{0}}^{t}g_{i,j}(r)w_{j}^{1+\beta_{i}}(r) dr,\ i=1,2,\ \left\lbrace j \right\rbrace=\left\lbrace1,2\right\rbrace/\{i\}.$$ Then $w_{i}(t) \geq x_{i}(t),\ i=1,2,$ for all $t \geq r_{0},$ and 
 \begin{align*}
  	\dot{x}_{i}(t)&=g_{i,j}(t)w_{j}^{1+\beta_{i}}(t) \geq g_{i,j}(t)x_{j}^{1+\beta_{i}}(t), \nonumber\\
  	\dot{h}_{i}(t)&=g_{i,j}(t)h_{j}^{1+\beta_{i}}(t). \nonumber  
\end{align*}
  Therefore, we have
  \begin{align*}
 	\left\{
 	\begin{aligned}
  	\dot{x}_{i}(t)-g_{i,j}(t)x_{j}^{1+\beta_{i}}(t) &\geq \dot{h}_{i}(t)-g_{i,j}(t)h_{j}^{1+\beta_{i}}(t), \nonumber\\
  	x_{i}(r_{0})=h_{i}(r_{0})&=r_{i}e^{-\gamma_{i}r_{0}}. \nonumber 
  \end{aligned} 
\right.
  \end{align*}
  It is enough to show that $h_{i}(t) \leq x_{i}(t),$ for all $ t \in [r_{0}, \tau).$ Suppose that $h_{i}(t) > x_{i}(t),$ for $t \in [r_{0}, r_{0}+\varepsilon),\ \varepsilon>0.$ Let $\Delta_{i}(t)=h_{i}(t)-x_{i}(t)$. Then an application of Taylor's formula yields 
  \begin{align}
  	\dot{\Delta}_{i}(t)&=\dot{h}_{i}(t)-\dot{x}_{i}(t)  \leq g_{i,j}(t)\left[h_{j}^{1+\beta_{i}}(t)-x_{j}^{1+\beta_{i}}(t)\right]  \nonumber\\
  	&\leq (1+\beta_{i}) \sup_{t \in[r_0,\tau]}|g_{i,j}(t)|\sup_{t \in[r_0,\tau]}|\theta h_{j}(t)+(1-\theta)x_{j}(t)|^{\beta_{i}} \Delta_{j}(t) \nonumber\\
  	&=\kappa_{i} \Delta_{j}(t). \nonumber 
  \end{align}	
  for $i=1,2,\ \left\lbrace j \right\rbrace=\left\lbrace1,2\right\rbrace/\{i\}$, where $\kappa_i=\sup\limits_{t \in[r_0,\tau]}|g_{i,j}(t)|\sup\limits_{t \in[r_0,\tau]}\left(| h_{j}(t)|+|x_{j}(t)|\right)^{\beta_{i}}$. Let $E(t)=\Delta_{1}(t)+\Delta_{2}(t),\ t \geq r_{0}.$ Then for $\kappa=\max\{\kappa_1,\kappa_2\}$, we have
  \begin{align}
  	\frac{dE}{dt} \leq \kappa_{1} \Delta_{1}(t)+\kappa_{2} \Delta_{2}(t) \leq \kappa  \left[\Delta_{1}(t)+\Delta_{2}(t) \right]=\kappa E(t)\Rightarrow 
  	\frac{d}{dt}\left(e^{-\kappa t} E(t) \right) \leq 0. \nonumber	
  \end{align}	 
  Therefore, we obtain $E(t) \leq 0$ for all $t \in [r_{0}, r_{0}+\varepsilon)$ and this implies that $\Delta_{2}(t)+\Delta_{1}(t) \leq 0,$ which is a contradiction, since $\Delta_i(t)>0,$ for all $t \in [r_{0}, r_{0}+\varepsilon)$.  Hence $w_{i}(t) \geq h_{i}(t),\ i=1,2,$ for all $t \in [r_{0},\tau).$
  
    \medskip\noindent\\
	{\bf Acknowledgements:} The first author is supported by the University Research Fellowship of Periyar University, India. M. T. Mohan would  like to thank the Department of Science and Technology (DST), India for Innovation in Science Pursuit for Inspired Research (INSPIRE) Faculty Award (IFA17-MA110). The third author is supported by the Fund for Improvement of Science and Technology Infrastructure (FIST) of DST (SR/FST/MSI-115/2016). The authors sincerely would like to thank the reviewers for their valuable comments and suggestions, which helped us to improve the manuscript significantly.

\end{document}